\documentclass[a4paper,12pt]{amsart}
\makeatletter
\@namedef{subjclassname@2020}{%
  \textup{2020} Mathematics Subject Classification}
\makeatother
\usepackage{latexsym}
\usepackage{amssymb}
\usepackage{graphicx}
\usepackage{wrapfig}
\usepackage{amsthm}
\usepackage{float}
\usepackage{epsfig}
\usepackage{multirow}
\usepackage{caption}
\usepackage[toc,page]{appendix}
\usepackage{old-arrows} 
\usepackage{lmodern, mathtools}
\usepackage{hyperref}
\RequirePackage{color}
\usepackage{tikz}
\usepackage{amscd,enumerate}
\usepackage{amsfonts}
\hypersetup{ colorlinks,
linkcolor=blue,
filecolor=green,
urlcolor=blue,
citecolor=blue }
\usepackage{helvet}
\usepackage{bigstrut}
\usepackage{microtype}
\usepackage{float, caption, booktabs, cellspace}
\definecolor{mypink1}{rgb}{0.158, 0.488, 0.178}
\definecolor{suput}{rgb}{0.058, 0.488, 0.08}
\definecolor{gris}{RGB}{138, 149, 151}

\usepackage{calrsfs}
\DeclareMathAlphabet{\pazocal}{OMS}{zplm}{m}{n}

\newtheorem{lemma}{Lemma}[section]
\newtheorem{corollary}[lemma]{Corollary}
\newtheorem{theorem}[lemma]{Theorem}
\newtheorem{proposition}[lemma]{Proposition}
\newtheorem{remark}[lemma]{Remark}
\newtheorem{definition}[lemma]{Definition}
\newtheorem{definitions}[lemma]{Definitions}
\newtheorem{example}[lemma]{Example}

\newtheorem{notation}[lemma]{Notation}
\newtheorem{corolary}[lemma]{Corollary}
\input xy
\xyoption{all}

\usepackage{array}
\newcolumntype{L}[1]{>{\raggedright\let\newline\\\arraybackslash\hspace{0pt}}m{#1}}
\newcolumntype{C}[1]{>{\centering\let\newline\\\arraybackslash\hspace{0pt}}m{#1}}
\newcolumntype{R}[1]{>{\raggedleft\let\newline\\\arraybackslash\hspace{0pt}}m{#1}}

\input xygraph

\newcommand{\bK}{\mathbb{K}}
\def\l{\lambda}
\def\a{\alpha}
\def\b{\beta}
\def\g{\gamma}
\def\o{\omega}

\def\GL{\mathop{\hbox{\rm GL}}}
\def\w{\omega}
\def\W{\mathcal{W} }
\def\B{\mathcal{B} }
\def\E{{A}}
\def\id{\mathop{\text{Id}}}
\def\f{\phi}
\newcommand{\Supp}{{\rm supp}}

\newcommand{\C}{{\mathcal{C}}}

\newcommand{\M}{{\mathcal{M}}}
\newcommand{\Ker}{{\rm{Ker}}}

\newcommand{\N}{{\mathbb{N}}}
\newcommand{\Z}{{\mathbb{Z}}}
\newcommand{\K}{{\mathbb{K}}}
\newcommand{\KN}{{\mathbb{K}^{\times}}}
\newcommand{\Q}{{\mathbb{Q}}}
\newcommand{\R}{{\mathbb{R}}}
\newcommand{\complex}{{\mathbb{C}}}
\def\Ad{\mathop{\rm{Adj}}\nolimits_1}
\def\supp{\mathop{\hbox{\rm supp}}}
\def\D{\mathop{\hbox{\rm D}}}
\def\Au{\mathop{\rm{Adj}}\nolimits_2}
\def\Av{\mathop{\rm{Adj}}\nolimits_3}
\def\Aw{\mathop{\rm{Adj}}\nolimits_4}
\def\remove#1{}

\newcommand{\uloopr}[1]{\ar@'{@+{[0,0]+(-4,5)}@+{[0,0]+(0,10)}@+{[0,0] +(4,5)}}^{#1}}

\def\endo{\mathop{\rm End}}

\def\GL{\mathop{\text{GL}}}

\def\aut{\mathop{\text{Aut}}}

\usepackage[utf8]{inputenc}
\usepackage{tikz}
\def\w{\omega}
\def\G{\mathcal G}

\def\esc#1{\langle#1\rangle}
\def\an{\mathop{\text{\rm ann}}}
\def\ann{\mathop{\text{\rm ann}}}
\def\im{\mathop{\text{\rm Im}}}
\def\rad{\mathop{\text{\rm rad}}}
\def\asi{\rm{asi}}
\def\ssi{\mathop{\text{\rm ssi}}}
\def\soc{\mathop{\text{Soc}}}
\def\sp{\mathop{\text{\rm span}}}
\def\esc#1{\langle #1\rangle}
\title{Chains in evolution algebras}

\author[Y. Cabrera]{Yolanda Cabrera Casado}
\address{Y. Cabrera Casado: Departamento de Matem\'atica Aplicada, E.T.S.Ingenier\'\i a Inform\'atica, Universidad de M\'alaga, Campus de Teatinos s/n. 29071 M\'alaga.   Spain.}
\email{yolandacc@uma.es}

\author[M. I. Gon\c calves]{Maria Inez Cardoso Gon\c calves}
\address{M. I. Cardoso Gon\c calves: Departamento de Matem\'atica, Universidade Federal de Santa Catarina, Florian\'opolis, SC, 88040-900 - Brazil}
\email{maria.inez@ufsc.br}

\author[D. Gon\c calves]{Daniel Gon\c calves}
\address{D. Gon\c calves: Departamento de Matem\'atica, Universidade Federal de Santa Catarina, Florian\'opolis, SC, 88040-900 - Brazil}
\email{daemig@gmail.com}

\author[D. Mart\'\i n]{Dolores Mart\'\i n Barquero}
\address{D. Mart\'\i n Barquero: Departamento de Matem\'atica Aplicada, Escuela de Ingenier\'\i as Industriales, Universidad de M\'alaga, Campus de Teatinos s/n. 29071 M\'alaga.   Spain.}
\email{dmartin@uma.es}

\author[C. Mart\'\i n]{C\'andido Mart\'\i n Gonz\'alez}
\address{C. Mart\'\i n Gonz\'alez: Departamento de \'Algebra Geometr\'{\i}a y Topolog\'{\i}a, Fa\-cultad de Ciencias, Universidad de M\'alaga, Campus de Teatinos s/n. 29071 M\'alaga.   Spain.}
\email{candido\_m@uma.es}

\thanks{ The first and the last two  authors are supported by the Junta de Andaluc\'{\i}a  through projects  FQM-336 and UMA18-FEDERJA-119 and  by the Spanish Ministerio de Ciencia e Innovaci\'on   through project  PID2019-104236GB-I00,  all of them with FEDER funds. The third author is supported by Conselho Nacional de Desenvolvimento Cient\'ifico e Tecnol\'ogico (CNPq) grant numbers 304487/2017-1 and 406122/2018-0  and Capes-PrInt grant number 88881.310538/2018-01 - Brazil.
}

\begin{document}


\begin{abstract}
In this work we approach three-dimensional evolution algebras from certain constructions performed on two-dimensional algebras. More precisely, we provide four different constructions producing three-dimensional evolution algebras from two-dimensional algebras. Also we introduce two parameters, the annihilator stabilizing index and the socle stabilizing index, which are useful tools in the classification theory of these algebras. Finally, we use moduli sets as a convenient way to describe isomorphism classes of algebras.

\end{abstract}

\maketitle

\section{Introduction}
A type of genetic algebras called evolution algebras appeared in \cite{VT}. This type of non-associative algebras arose in order to model Non-Mendelian genetics. For example, the classification into isotopism classes of three dimensional evolution algebras are used to describe the spectrum of genetic patterns of three distinct genotypes during a mitosis process, see \cite{FFN}.
It should also be noted that these algebras have multiple connections with other areas of mathematics, such as graph theory and stochastic processes, see \cite{random, EAG, volterra}. 

The search for invariants to classify algebras is one of the mainstreams in Mathematics. For example, the celebrated Elliot classification program for C*-algebras has attracted the interest of a generation of researches, with outstanding results. In the setting of evolution algebras, the classification of nilpotent algebras is studied in \cite {ElduqueLabra, HA, HA2}. The four dimensional perfect non-simple evolution algebras over a field with mild restrictions are classified in \cite{CKS}. The classification of general evolution algebras over finite fields is given in \cite{FF}. In the paper \cite{Imo}, the three-dimensional evolution algebras over a real field are classified. For three-dimensional evolution algebras over a general field, which verify certain restrictions, the classification is done in \cite {YC, CSV}, resulting in 116 non-isomorphic families. Given the huge number of different types of evolution algebras of dimension 3, it is necessary to classify them according to other characteristics, in order to categorize them in a more efficient and convenient manner. For example, in \cite{CV}, the evolution algebras of dimension three are classified according to properties such as irreducibility or degeneracy. Recently, a new branch of study has arisen in \cite{KS}. Its application in the setting of evolution algebras debuts in \cite{MDDCM}: the determination of evolution algebras with a faithful associative and commutative representation. So far, this idea has been pursued only in dimension $2$.

The philosophy of this work is to approach three-dimensional evolution algebras from certain constructions performed on two-dimensional algebras. We give four constructions producing three-dimensional evolution algebras from two-dimensional algebras. Also we use two parameters, $\asi(A)$ (Definition \ref{annstabin}) and $\ssi(A)$ (Definition \ref{ssi}), which help up to classify these algebras.  
We also use moduli sets (see Subsection~\ref{modulisets}) as a convenient way to describe isomorphism classes of algebras.

 Roughly speaking the study of three-dimensional evolution algebras falls into two disjoint classes:
 those with nonzero annihilator (sections \ref{secann} and \ref{annclas}); 
and the ones with zero annihilator (sections \ref{socchain} and \ref{classoc}). 
The classification of algebras $A$ with nonzero annihilator fall into $5$ disjoint cases depending on two parameters: (1) the values of the annihilating stabilizing index (abbreviated $\asi(A)$)  and (2) the value of $\dim(\an(A))$. In the case $\asi(A)=\dim(\an(A))=1$ all the algebras come from a construction $\Ad(B,\a)$ for a suitable evolution algebra $B$ of dimension two (see Theorem \ref{tornillo}). The construction $\Ad$ is described in Definition \ref{pcons}.

The ($3$-dimensional) algebras $A$ with zero annihilator have nonzero socle, so we can use $\dim(\soc(A))$ as one of the parameters for the classification task. The main results are Proposition~\ref{threedim},
in which $\dim(\soc(A))=3$, Theorem \ref{partwotwo}, Theorem~\ref{parthree} and Theorem \ref{carnosilla},
for the case $\dim(\soc(A))=2$, and Theorem \ref{lbnl} for the case $\dim(\soc(A))=1$. 
The algebras with $\dim(\soc(A))=3$ are either simple or direct sum of evolution algebras of dimension less or equal to $ 2$. 

The evolution algebras $A$ with $\dim(\soc(A))=2$ are separated in three large classes: (1) those whose socle has the extension property (Theorem \ref{partwotwo}); (2) those for which $\soc(A)=\sp(\{e_1,e_2+e_3\})$ with $\{e_i\}$ natural (Theorem \ref{parthree}); and (3) those for which $\soc(A)=\sp(\{e_1+e_2,e_2+e_3\})$ with $\{e_i\}$ natural (Theorem \ref{carnosilla}). In all the cases the construction $\Au$ defined in Definition~$\ref{adjdos}$ will play a fundamental roll. Finally, if $\dim(\soc(A))=1$ we describe these algebras in terms of two new  constructions:  $\Av$ and $\Aw$  (Subsection \ref{unidimsoc}), see Theorem \ref{lbnl} and Theorem \ref{macabel} depending on whether $\soc(A)^2\neq 0$ or $\soc(A)^2=0$, respectively.
\smallskip

At last, notice that a convenient way to describe the isomorphism classes of a collection of algebras is via a moduli set, that is, via an action of a group in a set such that the isomorphism classes of the algebras are in one-to-one correspondence with the orbits of the action. We will provide moduli sets for the three dimensional evolution algebras. In fact, the classification work we develop in this paper can be seen as the description of moduli sets for three-dimensional evolution algebras.

\section{Preliminaries}

An \emph{evolution algebra} over a field $\bK$ is a $\bK$-algebra $A$ which has a basis $\B=\{e_i\}_{i\in \Lambda}$ such that $e_ie_j=0$ for every $i, j \in \Lambda$ with $i\neq j$. Such a basis is called a \emph{natural basis}.
From now on, all the evolution algebras we will consider will be finite dimensional and  $\Lambda$ will denote a finite set $\{1, \dots, n\}$.

Let $A$ be an evolution algebra with a natural basis $\B=\{e_i\}_{i\in \Lambda}$.
Denote by $M_\B=(\omega_{ij})$ the \emph{structure matrix} of $A$ relative to $\B$, i.e., $e_i^2 = \sum_{j\in \Lambda} \omega_{ji}e_j$.

\begin{notation}\label{domingo}\rm 
We use the following notations in this paper:
\begin{enumerate} 
\item We denote by $\N$ the natural numbers including $0$, by $\N^*$ the set $ \N\setminus\{0\}$, by $\Z$ the integers and, if $\K$ is a field, then we denote by $\M_n(\K)$ the algebra of $n\times n$ matrices with coefficientes in $\K$. 
\item For  $n\in \N^{*}$, the notation $\K^n$ stands for the cartesian product 
$\K\buildrel{n}\over{\times\cdots\times}\K$, while the notation $\K^{\esc{n}}$ is used for 
$\K^{\esc{n}}:=\{\l^n\colon \l\in\K\}$ and 
$(\K^\times)^{\esc{n}}:=\{\l^n\colon \l\in\K^\times\}$.
\item For integers $1\le i,j\le n$, we denote by $E_{ij}$ the $n\times n$ matrix obtained by permuting the $i$th and $j$th rows of the identity matrix. For example, for $n=2$ we have $E_{12}=\tiny\begin{pmatrix}0 & 1\cr 1 &0\end{pmatrix},$ while for
$n=3$ we have \[E_{12}=\tiny\begin{pmatrix}0 & 1 & 0\cr 1 &0 & 0\cr 0 & 0 & 1\end{pmatrix}.\]
To be formal we should specify that the matrix size depends on the chosen $n$, but we will refrain from doing so in order to get a lighter notation. The reader should be able to deduce the value of $n$ from the context. 
\item We denote the cyclic group of order two by $\Z_2$, and
if $M$ is a matrix in $\M_n(\K)$ then we use matrix powers $M^i$, for $i\in\Z_2$, with the meaning  $M^0:=\id$ and $M^1:=M$.
\item Let $G$ be a group acting on a set $X$. We denote the set of orbits of $X$ under the action of $G$ by $X/G$.
\item Let $B$ be a $\K$-algebra. We denote by $\endo_{\K}(B)$ the algebra of linear maps from $B$ to $B$ and by $\endo(B)$ that of  endomorphisms of $B$ as algebra.
\item  We denote by $\hbox{Sym}^2(B)$ the $\K$-vector space of all symmetric bilinear forms on $B$.

\end{enumerate}
\end{notation}
\begin{definition}
\rm
Let $A$ be an evolution algebra,
$\B=\{e_i\}_{i\in \Lambda}$ a natural basis and $u=\sum_{i\in \Lambda}\alpha_ie_i$ an element of $A$.
The \emph{support of} $u$ \emph{relative to} $\B$, denoted $\Supp_\B(u)$, is defined as the set $\Supp_\B(u)=\{i\in \Lambda\ \vert \  \alpha_i \neq 0\}$.   If $X \subseteq A$, we put $\Supp_\B (X)= \cup_{x \in X} \; \Supp_\B (x)$. Following \cite{CSV1}, an evolution subalgebra $A'$ of an evolution algebra $A$ is said to have the \emph{extension property} if $A'$ has a natural basis which can be extended to a natural basis of $A$.
\end{definition}
\begin{remark}\label{enfadado}\rm
Assume that $\{e_i\}_{i\in I}$ is a natural basis of an evolution algebra $A$. 
A useful criterion for an element $z$  of $A$ with $z^2\neq 0 $ to be natural (i.e. embeddable in a natural basis) is that $\dim(\sp(\{e_i^2\colon i\in\Supp(z)\}))=1$ (see \cite[Teorema 3.3]{BCS}).
\end{remark}

We recall that an evolution algebra $A$ is non-degenerate if there exists a natural basis $\{e_i\}$ such that $e_i^2\neq 0$ for every $i$. In \cite[Corollary 2.19]{CSV1}, it is proved that this definition does not depend on the chosen basis since to be non-degenerated is equivalent to have $\ann(A)=0$.

With this in mind, we have

\begin{proposition}\label{gutten}
Let $A$ be a non-degenerate evolution algebra, then
\begin{enumerate}
\item Let $ \B=\{e_i\}_{i\in \Lambda}$ be a natural basis of $A$. If $z,z'$ are different elements such that $zz'=0$, then $\Supp(z)\cap\Supp(z')$ has cardinal different from $1$ (supports relative to $\B$).

\item If $A$ is a $3$-dimensional evolution algebra with natural basis $\{e_1,e_2,e_3\}$ and $\dim(A^2)=2$, then any other natural basis of $A$ is (up to permutations and nonzero multiples) of the form
$\{e_1+ke_2,e_1+k'e_2,e_3\}$, where $k,k'$ are different scalars in $\K^\times$.
\end{enumerate}
\end{proposition}
\begin{proof}
Consider $z,z'$  different and $zz'=0$, if the intersection of their supports has cardinal $1$ (say $i\in\Supp(z)\cap\Supp(z')$), then $0=e_i^2$ contradicting the fact that $A$ is non-degenerate. Let us prove the second assertion.
Since $\dim(A^2)=2$ any natural vector has support of cardinal less or equal to $2$. Indeed, suppose there is a natural element with support of cardinal 3. Then, by Remark~\ref{enfadado}, $\dim(\sp(\{e_i^2:i=1,2,3\}))=1$. This implies that $\dim(A^2)=1$, a contradiction.

A priori, the cardinals of the support of the three elements in a natural basis have the possibilities:
$$(2,2,2), (2,2,1), (2,1,1), (1,1,1).$$
Let us check that the first possibility, $(2,2,2)$, implies a contradiction. 
Since the three supports can not be equal, there are two of them whose intersection has cardinal one, which is impossible. The second possibility $(2,2,1)$ gives that the supports are  (up to permutations and scaling if necessary): $\{1,2\},\{1,2\},\{3\}$.
So the basis is as in the statement of the proposition. The third possibility 
implies again that the support of cardinal $2$ has intersection with some of the supports of cardinal one, a contradiction. The last possibility, $(1,1,1)$, implies that the basis is a reordering and scaling of the original one.
\end{proof}

\subsection{Moduli sets}\label{modulisets}
In this subsection we introduce the moduli sets that we will be used  to describe isomorphism classes of algebras, more concretely in our case   to describe isomorphism classes  of three-dimensional evolution algebras.
\begin{definition}\rm
For a class of $\K$-algebras $\C$, we will say that $(G,\M)$ is a {\it moduli set} for $\C$ if: (1) $G$ is a group, $\M$ is a $G$-set and (2) the set of isomorphism classes of algebras in $\C$ is in one-to-one correspondence with the orbits of $\M/G$, that is, orbits of $\M$ under the action of $G$.
\end{definition} 
 An easy example from the theory of evolution algebras is the following.
 \begin{example}\rm
 Consider the class of two-dimensional simple evolution algebras over a fixed field $\K$.
It is easy to check that for any algebra in this class there is a natural basis such that the structure matrix of the algebra, relative to this basis, is of one of the following forms:
\begin{align}
\hbox{Type I }&\qquad {\tiny \begin{pmatrix}1 & y\cr x & 1\end{pmatrix}}, \hbox{ with } x y-1\ne 0, x,y\in\K^\times.\cr
\hbox{Type II}&\qquad {\tiny \begin{pmatrix}0 & y\cr x & 1\end{pmatrix}}, \hbox{ with } x y\ne 0.\cr
\hbox{Type III}&\qquad {\tiny \begin{pmatrix}0 & y\cr x & 0\end{pmatrix}}, \hbox{ with } x y\ne 0.
\end{align}
Let us consider, for instance, the class of algebras of  type I. It can be checked that two algebras of type I are isomorphic if and only if either they have the same structure matrix or  the structure matrix of one of them is ${\tiny \begin{pmatrix}1 & y\cr x & 1\end{pmatrix}}$  and the other one is ${\tiny \begin{pmatrix}1 & x\cr y & 1\end{pmatrix}}$. Thus, we can define a moduli set $(\mathbb{F}_2,\M)$, where 
\[\M=\{{\tiny \begin{pmatrix}1 & x\cr y & 1\end{pmatrix}}\colon x y\ne 1, x,y\in\K^\times\}\]
and take the (multiplicative) group $\mathbb{F}_2=\{\pm1\}$
 acting on $\M$ in such a way that $1\in\mathbb{F}_2$ acts as the identity, while 
\[(-1)\cdot {\tiny \begin{pmatrix}1 & x\cr y & 1\end{pmatrix}}:={\tiny \begin{pmatrix}1 & y\cr x & 1\end{pmatrix}}.\]
Therefore the isomorphism classes of algebras of type I are in one to one correspondence with the $\mathbb{F}_2$-set $\M/\mathbb{F}_2$. In other words, the moduli set $(\mathbb{F}_2,\M)$ classifies the algebras of type I.
Furthermore, we can identify $\M$ with the 
Zarisky open subset 
\[\{(x,y)\in\K^2\colon x y\ne 1, x,y\in\K^\times\}\]
modulo the identification of each $(x,y)$ with $(y,x)$.

If we had $\K=\R$, then we could represent this $\mathbb{F}_2$-set in the real plane, removing the axis and the graphic of the function $y=1/x$, and identifying symmetric points relative to the line $y=x$:
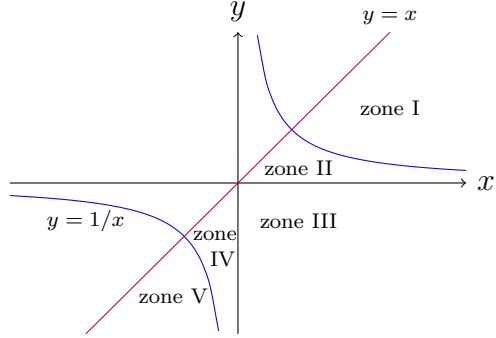
\begin{center}
\begin{tikzpicture}
      \draw[->] (-3,0) -- (3,0) node[right] {$x$};
      \draw[->] (0,-2) -- (0,2) node[above] {$y$};
      \draw[scale=0.5,domain=0.51:6,smooth,variable=\x,blue] plot ({\x},{2/\x});
      \draw[scale=0.5,domain=-6:-0.51,smooth,variable=\y,blue]  plot ({\y},{2/\y});%
      \draw[scale=0.5,domain=-4:4,smooth,variable=\x,purple] plot ({\x},{\x});
      \node[] at (-2,-0.5) {\tiny $y=1/x$};
      \node[] at (2,2.2) {\tiny $y=x$};
      \node[] at (2,1) {\tiny zone I};
      \node[] at (.8,.2) {\tiny zone II};
      \node[] at (.8,-.5) {\tiny zone III};
      \node[] at (-.3,-.7) {\tiny zone};
      \node[] at (-.2,-1) {\tiny IV};
      \node[] at (-.85,-1.5) {\tiny zone V};
    \end{tikzpicture}
\end{center}
so that we would have $5$ zones and each point representing an isomorphism class of algebras of type I. The fact that the set consists of $5$ connected components has to do with the existence of $5$ homotopy classes of algebras of type I. However we will not pursue this homotopy ideas further.
\end{example}
Complementing our comment about moduli sets in the introduction, we observe that once we have constructed a moduli set $(\M,G)$ for a class of algebras $\C$, we have parametrized the isomorphism classes of $\C$ by the orbits of $\M/G$ and so this is a useful tool for classification tasks. Next we describe some of the moduli sets that will be used in our work (some of the moduli set that will appear require further background to be described, and so we will introduce then as needed). 

\subsubsection*{$(\D_2(\K)\rtimes\Z_2,\K^2\setminus\{0\})$:} Among the moduli sets used in the classification of three-dimensional evolution algebras, we will use the group of non uniform scales jointly with the symmetries defined below. Roughly speaking, this moduli set consists of the  diagonal group of $2\times 2$ invertible matrices plus a symmetry. 

More precisely, let $\D_2(\K)=\tiny\left\{ \begin{pmatrix} \lambda_1 & 0 \\ 0 & \lambda_2 \end{pmatrix}: \, \, \lambda_i \in \KN\right\}$ be the group of diagonal matrices. 
Consider the subgroup of $\GL_2(\K)$ generated by $\D_2(\K)$ and $ E_{12}=\tiny\begin{pmatrix}0 & 1\cr 1 & 0\end{pmatrix}$, which we will denote by
$\D_2(\K)\rtimes\Z_2$. We have 
\[\D\nolimits_2(\K)\rtimes\Z_2=\left\{\begin{pmatrix}\l_1 & 0\cr 0 & \l_2\end{pmatrix}\colon \l_i\in\K^\times \right\}\sqcup \left\{\begin{pmatrix}0 & \l_1\cr \l_2 & 0\end{pmatrix} : \l_i \in \K^{\times}\right\}\]

This is a subgroup of the general linear group $\GL_2(\K)$, therefore there is an induced representation $(\D_2(\K)\rtimes\Z_2)\times (\K^2\setminus\{0\})\to (\K^2\setminus\{0\})$ such that, 
for $M\in \D_2(\K)\rtimes\Z_2$ and 
$v\in \K^2\setminus\{0\}$ (as column vector), the action $M v$ is the usual multiplication. 
So we have a moduli set $(\D_2(\K)\rtimes\Z_2,\K^2\setminus\{0\})$.
Observe that two vectors $v,v'\in\K^2\setminus\{0\}$ are in the same orbit under the action of $\D_2(\K)\rtimes\Z_2$ if and only if the number of zero entries in $v$ coincide with the number of zero entries in $v'$. A set of representatives of the orbits are $(1,0)$ and $(1,1)$.

\subsubsection*{$((\K^\times)^{\esc{2}},\K^\times)$:} This moduli set is constructed by considering  the group 
$(\K^\times)^{\esc{2}}:=\{k^2\colon k\in\K^\times\}$ and its natural action on the set $\K^\times$
given by 
\begin{equation}\label{radiola}
(\K^\times)^{\esc{2}}\times\K^\times\to\K^\times,
\end{equation}
such that for any $g\in (\K^\times)^{\esc{2}}$ and 
$\l\in\K^\times$ we have $g\cdot\l=g\l$. Note that the cardinal of the set of orbits $\K^\times/(\K^\times)^{\esc{2}}$ depends greatly of the nature of the ground field $\K$. For instance, if $\K$ is algebraically closed it has cardinal one. If $\K=\R$, then $\K^\times/(\K^\times)^{\esc{2}}$ has cardinal $2$ and if $\K=\Q$, then there are countable many orbits.

\section{Annihilator chain}\label{secann}

In this section we will define the construction of new algebras by the procedure of adjunction of type one. We will study the isomorphism problem for this class of algebras. In order to do this work we
use the upper annihilating series and we define the annihilator stabilizing index of an algebra. We will apply all this tools and results in the next section  to classify the three-dimensional evolution algebras with non-zero annihilator.

We recall that given a nonassociative algebra $A$, we have the following sequences of subspaces:
\begin{align*}
A^0 & = 0, & A^1=A,  &  &A^{k+1}=\sum_{i=1}^{k}A^i A^{k+1-i} \, \, {\rm for} \,\, k >1.
\end{align*}

If $J$ is an ideal of $A$, we consider $J$ as an algebra and the powers of $J$ are defined as in the previous case. An algebra (ideal) $A$ is {\it nilpotent} if there exist $n \in \N^*$ such that $A^n=0$. 

\begin{definition}\label{tamandua}\rm
Let $A$ be an algebra and denote by $\mathcal{P}(A)$ the power set of $A$.
We define the map $\rho_A\colon \mathcal{P}(A) \to \mathcal{P}(A)$, where
$\rho_A(S)=SA \,\cup\, AS=\{sa \colon s \in S, a \in A\} \cup \{as \colon s \in S, a \in A\}$ for any $S \in \mathcal{P}(A)$. We denote  $\rho_A(x):=\rho_A(\{x\})$. 
\end{definition}
In our classification results we will make use of the upper annihilating series, as defined in \cite[Definition 3.3]{ElduqueLabra}. We recall this definition bellow.
\begin{definition}\rm
Let $A$ be an algebra. We define $\an^{(0)}(A): =\{0\}$   and $\an^{(i)}(A)$  in the following way 
\[\an\nolimits^{(i)}(A)/\an\nolimits^{(i-1)}(A): =\an\nolimits(A/ (\an\nolimits^{(i-1)}(A)).\]
The chain of ideals:
\[\{0\}=\an\nolimits^{(0)}(A)\subseteq \an\nolimits^{(1)}(A)\subseteq \cdots \subseteq \an\nolimits^{(i)}(A)\subseteq \cdots\]
is called the {\it upper annihilating series}.
\end{definition}

Observe that $\an^{(1)}(A)=\an(A)=\{x \in A : xA=Ax=0\}$.

Next we define the key classifying parameter in this section, namely the annihilator stabilizing index.

\begin{definition}\label{annstabin}\rm
Let $A$ be an algebra. If there exists $k$ such that $k=\text{min}\{ q \colon \an^{(q)}(A)=\an^{(q+1)}(A)\}$, then we call it the \emph{annihilator stabilizing index} of $A$, denoted by $\asi(A)$.
\end{definition}

\begin{remark}
 Observe that if $A$ is finite dimensional then $\asi(A)$ always exists. If the chain of annihilators does not stabilize, we write $\asi(A)=\infty$.
\end{remark}

Using the map $\rho_A$ of Definition~\ref{tamandua} we obtain the following useful description of $\an^{(i)}(A)$.

\begin{lemma}\label{zuru}
Let $A$ be an algebra. Then:
\begin{enumerate}[\rm (i)]
    \item \label{zuru1} $\an^{(i)}(A)=\{x\in A \colon xA \, \cup \, Ax\subset \an^{(i-1)}(A)\}$ for $i\ge 1$.
    \item \label{zuru2} $\an^{(i)}(A)=\{x\in A \colon \rho_A^{k}(x)  \subset \an^{(i-k)}(A) \text{ for all } k\in\{0,1,\dots ,i\}\}$.
\end{enumerate}
\end{lemma}
\begin{proof}

First we prove \eqref{zuru1}. If $a \in \an^{(i)}(A)$, then $aA \, \cup \, Aa \subset \an^{(i-1)}(A)$ by definition. For the other inclusion, notice that if $aA \, \cup \, Aa \subset \an^{(i-1)}(A)$, then $\bar{a}\left(A/\an\nolimits^{(i-1)}(A)\right)=\bar{0}$ and $\left(A/\an^{(i-1)}(A)\right)\bar{a}=\bar{0}$, so  $ \bar{a}\in \an\left(A/\an\nolimits^{(i-1)}(A)\right)= 
\an\nolimits^{(i)}(A)/\an\nolimits^{(i-1)}(A).$
Therefore $a\in \an^{(i)}(A)$.
For \eqref{zuru2}, first recall that if $k=0$, then  $\rho_A^0$ is the identity map. For $k=1$, if $x \in \an^{(i)}(A)$ then we have by definition that $\rho_A(x) \subset \an^{(i-1)}(A)$. Hence  $\rho_A^2(x)\subset \an^{(i-2)}(A)$ (by \eqref{zuru1}) and iterating $k$ times we obtain that $\rho_A^k(x)\subset \an^{(i-k)}(A)$. For the other inclusion, let $x \in A$ be such that $\rho_A^k(x)\subset \an^{(i-k)}(A)$ for all $k\in\{0,1, \dots , i\}$. In particular, for $k=0$ we get  $x \in \an^{(i)}(A)$. \end{proof}

Before we proceed to the next subsection we study the relation between the absorption radical of an algebra $A$ and its upper annihilating series. We start with the definition of the absorption property and absorption radical.

\begin{definition} \rm
Let $I$ be an ideal of an algebra $A$. We say that $I$ has the \emph{absorption property} if $xA \, \cup \, Ax \subset I$ implies $ x \in I$. The \emph{absorption radical} of $A$ is the intersection of all ideals of $A$ having the absorption property, denoted by $\rad(A)$.
\end{definition}

\begin{proposition}\label{musiquita}
Let $A$ be an algebra.
\begin{enumerate} [\rm (i)]
    \item \label{hig1} $x\in \an^{(k)}(A)$ if and only if $\rho_A^k(x)=0$.
    \item \label{hig2}$\an^{(k)}(A)\subseteq  \rad(A)$ for any $k$.
    \item\label{hig3} If there exists an annihilator stabilizing index $k$, then $\an^{(k)}(A)$ is an absorption ideal. Therefore  $\rad(A)=\an^{(k)}(A)$. 
    \item\label{hig4}$\an^{(i)}(A)$ is a nilpotent ideal of $A$ for all $i$.
       \end{enumerate}
\end{proposition}
\begin{proof}
For \eqref{hig1} we take $k=i$ in Lemma \ref{zuru} \eqref{zuru2}. 

\eqref{hig2}  We will prove that 
$\an^{(k)}(A)\subset I$ for any absorbent ideal $I$ of $A$.
This is trivially true for $k=0$. Assume that 
$\an^{(0)}(A),\ldots, \an^{(k-1)}(A)$ are contained in $I$.
We prove that $\an^{(k)}(A)\subset I$. For this, let $x\in \an^{(k)}(A)$. Then, by item (\ref{zuru1}) in Lemma~\ref{zuru}, we have that $xA$ and $Ax$ are contained in $\an^{(k-1)}(A)\subset  I$. Since $I$ is absorbent we conclude that $x\in I$.

 \eqref{hig3}
 Let $k:=\asi(A)$. We prove that  $\rad(A)\subset \an^{(k)}(A)$. For this, it suffices to prove that $\an^{(k)}(A)$ is absorbent.
 Suppose that $xA \, \cup \, Ax \subseteq \an^{(k)}(A)$. We have to prove that $x \in \an^{(k)}(A)$.  Since $xA \, \cup \, Ax \subseteq \an^{(k)}(A)$, we have that $\bar x \in \an \left( A/ \an^{(k)}(A)\right)=\an^{(k+1)}(A)/\an^{(k)}(A)$
$= \left\{ \bar 0 \right\},$ which implies that $x \in \an^{(k)}(A)$. Hence $\rad(A) \subseteq \an^{(k)}(A)$. By \eqref{hig2} we obtain that $\rad(A)=\an^{(k)}(A)$.

\eqref{hig4} Let $J=\an\nolimits^{(i)}(A)$ and observe that $J \supseteq J^2 \supseteq J^3 \supseteq \ldots$. Moreover $J^2=JJ \subseteq \an\nolimits^{(i-1)}(A)$ by \eqref{zuru1}. Then $\an\nolimits^{(i-1)}(A) \supseteq J^2 \supseteq J^3 \supseteq J^4 \supseteq \ldots$, so $\an\nolimits ^{(i-2)}(A)\supseteq J^4=JJ^3+J^2J^2+J^3J \supseteq J^5 \supseteq J^6 \ldots$ Reasoning in the same way we get $\an\nolimits ^{(i-\alpha)}(A)\supseteq J^{2^\alpha} $. In particular, for $\alpha=i$ we have that  $J^{2^i}=0 $.
\end{proof}

\begin{remark} The absorption radical in an evolution algebra is a basic ideal, see \cite[Proposition 3.1]{YE}.
\end{remark}

\subsection{Adjunction of type one.}

In this subsection we will study how to construct an evolution algebra of dimension $n+1$ by adjunction of an annihilator element to an $n$-dimensional evolution algebra with a symmetric bilinear form. 
\begin{definitions}\label{pcons}\rm 
Let $\E$ be an evolution $\K$-algebra endowed with a  symmetric bilinear form $\a\colon \E\times\E\to \K$ such that $\a$ diagonalizes with respect to a natural basis of $\E$. Such a symmetric bilinear form will be called {\em compatible symmetric bilinear form}.
We will say that $(\E,\a)$ is a {\em diagonalizable evolution algebra} if it is endowed  with a compatible symmetric bilinear form $\a$. We will drop the bilinear form $\a$ if it is clear from the context.
If $\E$ is a diagonalizable evolution $\K$-algebra we can define a new evolution algebra $\K\times\E$ with product 
\[(\l,x)(\l',x')=(\a(x,x'),xx').\]
\end{definitions}
The fact that $\K\times\E$ is an evolution algebra is easily seen considering a natural basis $\{e_i\}_{i\in I}$ of $\E$ which also diagonalizes $\a$. Then defining $f_0:=(1,0)$ and $f_i=(0,e_i)$ for $i\in I$, we get 
a natural basis $\{f_i\}_{i\in I\cup\{0\}}$ of
$\K\times\E$. The algebra $\K\times\E$ will be called the {\em adjunction of an annihilating element to the diagonalizable evolution algebra $(\E,\a)$}. This algebra will be denoted $\Ad(\E,\a)$.
Two diagonalizable evolution algebras are \emph{isometrically isomorphic} if there exists an isomorphism of evolution algebras which preserve the symmetric bilinear forms.

\begin{remark}\label{tasha}
If B is a diagonalizable evolution algebra then $\dim (\an (\Ad(B,\a)))= 1 + \dim (\an (B))$.
\end{remark}

\begin{remark} \label{desayuno}
 \rm
 Let $(\E_1, \alpha_1$) and $(\E_2,\alpha_2$) be two diagonalizable evolution algebras. If there exists an isometric isomorphism $f\colon\E_1 \to \E_2$, then $\Ad(\E_1,\alpha_1) \cong \Ad(\E_2,\a_2)$.
\end{remark}

\begin{lemma}\label{chiringuito}
Let $\E$ be  an  evolution algebra  with $\dim(\an(\E))=1$. Then  $\E\cong\Ad(B,\a)$, where $B:=\E/\an(\E)$ and $\a$ is a compatible bilinear form in $B$. Furthermore, if  $\asi(\E)=1$ then $\an(B)=0$.
\end{lemma}

\begin{proof}
We can write $\ann(\E)=\K z_0$ for a certain $z_0 \in \E$. We will define a compatible scalar product in $B$,  and then we will prove that $\E\cong\Ad(B,\a)$ for a suitable $\a$.

First, we have $\E=\an(\E)\oplus C$ for some subspace $C$ which can be chosen to have a basis of natural vectors. If $y_1,y_2\in C$
we have $y_1y_2=\theta(y_1,y_2)z_0+q(y_1,y_2)$,
where $\theta\colon C\times C\to \K$ is a symmetric bilinear form and 
$q\colon C\times C\to C$ is symmetric and bilinear. Then, for any two elements $a_1,a_2\in\E$ we have $a_i=k_i z_0+c_i$ (for $i=1,2$) and
\[a_1a_2=(k_1z_0+c_1)(k_2z_0+c_2)=\theta(c_1,c_2)z_0+q(c_1,c_2).\]
Thus, the only thing needed to prove that  $\E=\Ad(C,\theta)$ is that $C$ is an evolution algebra relative to $q$ and that $\theta$ diagonalizes in some natural basis of $C$. In order to do that, take any natural basis $\{e_i\}_{i\in\Lambda}$ of $\E$ and write $e_i=k_i z_0+y_i$ for any $i \in \Lambda$. Then, the set $\{y_i\}_i$ is a system of generators of the vector space $C$, because for any $c\in C$ we have $c=\sum_i h_i e_i$ (for some scalars $h_i\in\K$)
 and hence $c=\sum_i h_i k_i z_0+\sum_i h_i y_i$. Thus $\sum h_ik_i=0$ and
 $c=\sum h_i y_i$. On the other hand the vectors $y_i$'s are pairwise orthogonal: if $i\ne j$ we have $0=e_ie_j=y_iy_j$. Then, there is a basis $\B=\{y_{i_1},\ldots,y_{i_q}\}$ of $C$ which satisfies 
 $y_{i_n}y_{i_m}=0$ for $i_n\ne i_m$. Consequently 
 \[0=y_{i_n}y_{i_m}=\theta(y_{i_n},y_{i_m})z_0+q(y_{i_n},y_{i_m})\Rightarrow\  \theta(y_{i_n},y_{i_m})=0,\ q(y_{i_n},y_{i_m})=0.\] 
This proves that $C$ is an evolution algebra for the product $q$ with natural basis $\B$ and furthermore $\theta$ diagonalizes in $\B$. Also 
\begin{equation}\label{grinch}
\E=\Ad(C,\theta).\end{equation}
On the other hand, the map $f\colon C\to B=\E/\an(\E)$ such that $f(y)=\bar y$ (the class of $y$ modulo $\an(\E)$) is an isomorphism of vector spaces and for any $c_1,c_2\in C$
\[f(c_1c_2)=f(\theta(c_1,c_2)z_0+q(c_1,c_2))=\overline{q(c_1,c_2)}=\overline{c_1}\ \overline{c_2}=f(c_1)f(c_2).\] So $C$ is isomorphic as an evolution algebra to $\E/\an(\E)=B$. If we define in $B$
the unique compatible symmetric bilinear form $\a$  that makes of $f$ an isometric isomorphism then, applying Remark~\ref{desayuno}, we have $\Ad(C,\theta)\cong\Ad(B,\a)$, from which we conclude that 
\begin{equation}\label{grinchdos}
\E\cong\Ad(\E/\an(\E),\a).\end{equation}
For the last assertion in the statement of the Lemma, observe also that $\an(\E/\an(\E))=0$ because $\asi(\E)=1$.

\end{proof}

\begin{example}\label{contraex}
\rm
Consider the evolution algebras $\E_1$ and $\E_2$ with natural bases $\{e_1, e_2, e_3\}$ and $\{f_1, f_2, f_3\}$ and product relative to these bases given by the matrices

\[\left(\begin{matrix}
0 & 0 & 0 \\
0 & 1 & 1 \\
0 & 0 & 0
\end{matrix}\right) \quad \hbox{and} \quad\left(\begin{matrix}
0 & 1 & 0 \\
0 & 1 & 1 \\
0 & 0 & 0
\end{matrix}\right) \text{  respectively}.\]

These evolution algebras satisfy that $\asi(\E_i) =1$, $\dim(\an(\E_i))=1$, and $\E_1/\an(\E_1) \cong \E_2/\an(\E_2)$ but they are not isomorphic.
To prove this statement, note that $\E_1= \an(\E_1) \oplus I$, where $I$ is the ideal generated by $\{e_2, e_3\}$, i.e., it is a reducible evolution algebra, but this is not the case for $\E_2$.
Assume that there exists an evolution ideal $J$ in $\E_2$ such that $\E_2= \an(\E_2) \oplus J$. Denote by $\{\alpha_1 f_1 + \alpha_2 f_2 + \alpha_3 f_3, \beta_1 f_1 + \beta_2 f_2 + \beta_3 f_3\}$ a natural basis of $J$. Then $(\alpha_1 f_1 + \alpha_2 f_2 + \alpha_3 f_3)(\beta_1 f_1 + \beta_2 f_2 + \beta_3 f_3) = 0$ implies $\alpha_2=0, \beta_3=0$ or $\alpha_3=0, \beta_2=0$. Assume the first case (the other one is analogue). Since $J$ is a subalgebra, $(\alpha_1f_1+\alpha_3f_3)^2\in J$, i.e., 
\[\alpha_3^2f_2 = x (\alpha_1f_1+\alpha_3f_3) + y (\beta_1 f_1 +\beta_2 f_2).\] This implies $x\alpha_3=0$ and, since $\alpha_3 \neq 0$ (else $ \alpha_1 f_1 + \alpha_2 f_2 + \alpha_3 f_3 \in \an(\E_2) $), necessarily $x=0$ and consequently $y\beta_1=0$ and $y\beta_2=\alpha_3^2$. We know that $\alpha_3\neq 0$ hence 
 $y \neq 0$ and so $\beta_1=0$. 
Use again that $J$ is a subalgebra to obtain 
$(\beta_2f_2)^2\in J$, i.e. 
$\beta_2^2(f_1+f_2)= x' (\alpha_1f_1+\alpha_3f_3) + y' \beta_2 f_2$, that is $x'\alpha_3=0$, implying again $x'=0$ and, consequently, $\beta_2=0$, a contradiction.
\end{example}

\subsection{Isomorphisms between adjunction algebras of type one.}

In this subsection we study the relation between isomorphism between evolution algebras and isomorphism of their adjunctions. We start by showing that for a diagonalizable evolution algebra a scaling of the associated bilinear form yields isomorphic adjunctions.

Let $B$ be a $\K $-algebra  with $\alpha\colon B\times B\to \K$ and consider a new inner product $\beta\colon B\times B\to \K$ given by $\beta(x,y)=k\alpha(x,y)$ for a fixed nonzero $k\in \K$. 
Then denote $B_\a:= \K\times B$ with the product $(\l,x)(\l',x')=(\a(x,x'),xx')$ and $B_\b:=\K\times B$
with the product $(\l,x)(\l',x')=(\b(x,x'),xx')$. Observe that the map $F\colon B_\a\to B_\b$ such that 
$F(\l,x)=(k\l,x)$ is an isomorphism of $\K$-algebras. Indeed:
\[\begin{array}{lll}
    F((\lambda,x)(\mu,y)) & = & F(\alpha(x,y),xy)=(k\alpha(x,y),xy)=(\beta(x,y),xy) \\
     & = &(k\lambda,x)(k\mu,y)=F(\lambda,x)F(\mu,y).
\end{array} \]

In particular, when $(B,\a)$ is a diagonalizable evolution algebra, $(B,\b)$ is also a diagonalizable evolution algebra and $\Ad(B,\alpha)$ and $\Ad(B,\beta)$ are isomorphic, via the isomorphism $F\colon \Ad(B,\alpha)\to \Ad(B,\beta)$ such that $F(\lambda,x)=(k\lambda,x)$. Hence \[\Ad(B,\alpha)\cong\Ad(B,k\alpha)\] for $k\in \K^\times:=\K\setminus\{0\}$.\medskip

Next we prove that if two diagonalizable evolution algebras are isometrically isomorphic then their adjunctions are isomorphic.

\begin{proposition}\label{jejeje} 
Let $(B_i,\a_i)$  ($i=1,2$) be two diagonalizable evolution algebras and let $\beta\colon B_1\to B_2$ be an algebra isomorphism $\beta\colon B_1\to B_2$. Assume also that there is $\f\in B_1^*:=\hom_k(B_1,\K)$ (the usual dual space) satisfying \[\f(xy)=\a_2(\beta(x),\beta(y))-\a_1(x,y)\] for any $x,y\in B_1$. Then the map $F\colon\Ad(B_1,\a_1)\to \Ad(B_2,\a_2)$ such that 
$F(\lambda,x):=(\lambda+\f(x),\beta(x))$ is an algebra isomorphism.
\end{proposition}
\begin{proof}

It is easy to check the linearity of $F$ and its bijective character. Furthermore:
 \[F((\lambda,x)(\mu,y))=F(\a_1(x,y),xy)=(\a_1(x,y)+\f(xy),\beta(xy))=(\a_2(\beta(x),\beta(y)),\beta(x)\beta(y))=\]
\[(\lambda+\f(x),\beta(x))(\mu+\f(y),\beta(y))=F(\lambda,x)F(\mu,y).\]
\end{proof}

\begin{remark}\rm
 In particular if $\beta\colon B_1\to B_2$ is an isometric isomorphism then, taking $\f=0$, we get that $\Ad(B_1,\a_1)\cong \Ad(B_2,\a_2)$.
\end{remark}

In our next result we describe when an isomorphism between adjunctions imply an isormophism between the algebras. For this, recall that if $B$ be an evolution algebra with zero annihilator, and with a compatible inner product $\a$, then $\hbox{ann}(\Ad(B,\a))=\K\times \{0\}$ (this follows from Remark~\ref{tasha}).

\begin{proposition}\label{dejame}
Assume that $(B_i,\a_i)$ are two diagonalizable evolution algebras and  $B_2$
has zero annihilator. Assume that $F\colon\Ad(B_1,\a_1)\to \Ad(B_2,\a_2)$ is an isomorphism.
Then, scaling the inner product of $B_2$ if necessary, we have that:\par 
\item{1)} There is an isomorphism of algebras $\beta\colon B_1\to B_2$.
\item{2)} There is an element $\f\in B_1^*$ such that $\f(xy)=\a_2(\beta(x),\beta(y))-\a_1(x,y)$ for any $x,y\in B_1$.
\end{proposition}
\begin{proof}
Let $F\colon \K\times B_1\to \K\times B_2$ be as in the statement of the proposition. Then, up to scalar multiples, we have that $F((1,0))=(1,0)$ (because $(1,0)\in\hbox{ann}(Ad(B_1,\a_1))$ implies $F((1,0))\in\hbox{ann}(Ad(B_2,\a_2))=\K\times 0$). So we have $F((\lambda,0))=(\lambda,0)$ for any scalar $\lambda$.
Now, $F((0,x))=(\f(x),\beta(x))$ for some linear maps $\f\colon B_1\to \K$ and $\beta\colon B_1\to B_2$.  We prove that $\beta$ is
an monomorphism:  if $\beta(x)=0$ then $F((0,x))=(\f(x),0)=F((\f(x),0))$ and, since $F$ is an isomorphism, $(0,x)=(\f(x),0)$ implies $x=0$.
Also $\beta$ is epimorphism, since for any $y\in B_2$ we have $(0,y)=F((\lambda,x))$ for some $\lambda\in \K$ and $x\in B_1$. So $y=\beta(x)$.
Next we check that $\beta(xy)=\beta(x)\beta(y)$ for any $x,y$, and simultaneously we check condition 2) in the proposition.
For this, notice that
\[F((0,x)(0,y))= F((\a_1(x,y),xy))=(a_1(x,y)+\f(xy),\beta(xy)), \text{ while }\] 
\[F((0,x))F((0,y))=(\f(x),\beta(x))(\f(y),\beta(y))=(\a_2(\beta(a),\beta(y)),\beta(x)\beta(y)).\]
So we get $\beta(xy)=\beta(x)\beta(y)$ and $\f(xy)=\a_2(\beta(x),\beta(y))-\a_1(x,y)$ for any $x,y\in B_1$.
\end{proof}

\begin{corolary}\label{hambre} Let $B_i$, $i=1,2$ be diagonalizable perfect evolution algebras. Then we have  $\Ad(B_1,\a_1)\cong\Ad(B_2,\a_2)$ if and only if $B_1\cong B_2$.
\end{corolary}
\begin{proof}
Since the perfection of a finite dimensional evolution algebra implies that its annihilator is zero, Proposition \ref{dejame} gives that 
$\Ad(B_1,\a_1)\cong\Ad(B_2,\a_2)$ implies $B_1\cong B_2$.
Reciprocally, assume $\beta\colon B_1\rightarrow B_2$ is an isomorphism. 
Let $\{u_i\}$ be a natural basis of $B_1$ (necessarily it diagonalizes $\a_1$) and notice that the (natural) basis $\{\b(u_i)\}$ of $B_2$ diagonalizes $\a_2$.
Let $(\w_i^j)$ be the structure matrix of $B_1$ (so $u_i^2=\sum_j\w_i^ju_j$ for any $i$). Let $(\tilde\w_i^j)$ be the inverse matrix of $(\w_i^j)$, that is, $\sum_j\w_i^j\tilde\w_j^k=\delta_i^k$ (Kronecker delta) for any $i,k$.
Define $\f\colon B_1\to \K$ by writing $\f(u_k):=\sum_i \tilde\w_k^i(\a_2(\b(u_i),\b(u_i))-\a_1(u_i,u_i))$ for any $k$.
Then it is easy to check that $\f(u_iu_j)=\a_2(\beta(u_i),\beta(u_j))-\a_1(u_i,u_j)$ for any $i,j$, whence $\f(xy)=\a_2(\b(x),\b(y))-\a_1(x,y)$ for any $x,y\in B_1$, and the result follows from Proposition~\ref{jejeje}.
\end{proof}

\begin{remark}\label{pesado} \rm

A moduli set for algebras of type $\Ad(B,\a)$ is the following. Fix an evolution algebra $B$ and consider the group $\G=\aut(B)\times B^*$ (where $B^*$ is the dual space of $B$) endowed with the product \[(\eta,S)(\theta,T)=(\eta\theta,S\theta+T),\]
where $\eta,\theta\in\aut(B)$ and $S,T\in B^*$.
Define $\mathcal U$ to be the $\K$-space of all compatible symmetric bilinear forms $\a\colon B\times B\to\K$. There is an action $\G\times{\mathcal U}\to {\mathcal U}$ given by 
\begin{equation}\label{crepus}
    (\theta,T)\a:=\a',\hbox{ where } \a'(\theta(x),\theta(y))-\a(x,y)=T(xy),
\end{equation}
for any $x,y\in B$.
Then $(\G,\mathcal{U})$ is a moduli set, which will be used to classify algebras of type  $\Ad(B,\a)$ (see Proposition~\ref{dejame}).
\end{remark}

\section{Classification in terms of the upper annihilating series}\label{annclas}
In this section we will study degenerate three-dimensional evolution algebras $\E$ in terms of their upper annihilating series. 
Before we state the classification theorem we prove the following lemma.

\begin{lemma}\label{starving}
Let $A$ be an evolution algebra with natural basis $\B=\{e_i\}$, $\dim(\ann{(A)})=1$ and product $e_1^2=0$, $e_2^2=\a e_1$, $e_3^2= \b e_1+ \gamma e_2 + \delta e_3$ with $\gamma \neq 0$ or $\delta \neq 0$. Then there exists another natural basis $\{f_i\}$ such that $f_1^2=0$, $f_2^2=f_1$ and $f_3^2 \in \sp(\{f_2,f_3\})$.
\end{lemma}

\begin{proof}
First, scaling $e_1$ we may consider without lost of generality that $\a=1$. Now, we take the natural basis $\{f_i\}$ with $f_1=e_1$, $f_2=xe_1+e_2$ and $f_3=x'e_1+e_3$. Note that we can choose $x$ and $x'$ such that $f_3^2 \in \sp(\{f_2,f_3\})$.
\end{proof}

We have the following classification:

\begin{theorem} \label{tornillo}
Let $\E$ be a three-dimensional evolution algebra over a field $\K$ and assume that $\an(\E)\ne \{ 0\}$. Then we have that one, and only one, of the following possibilities holds: 
\begin{enumerate}
    \item If $\asi(\E)=3$ then $\E$ is a nilpotent evolution algebra.
    
\item If $\asi(\E)=2$, and $\an(\E)$ has dimension $1$, then $\E$ is isomorphic to the evolution algebra with structure matrix
$\tiny\begin{pmatrix}
0 & 1 & 0\cr 0 & 0 & 0\cr 0 & 0 & 1 \end{pmatrix}$, or with structure matrix $\tiny\begin{pmatrix}
0 & 1 & 0\cr 0 & 0 & 1\cr 0 & 0 & 1\end{pmatrix}$, or with structure matrix $\tiny\begin{pmatrix}
0 & 1 & \b \cr 0 & 0 & 0\cr 0 & 0 & 0\end{pmatrix}$ with $\b\neq 0$ (notice that the first two algebras are non-isomorphic and also non-isomorphic to any algebra of the third type). 
Furthermore, the algebras with the previous structure matrix (depending on $\b$), are classified by the moduli set \eqref{radiola}.

\item If $\asi(\E)=2$, and $\an(\E)$ has dimension $2$, then $\E$ is isomorphic to the nilpotent evolution algebra with structure matrix
$\tiny\begin{pmatrix}
0 & 0 & 1\cr 0 & 0 & 0\cr 0 & 0 & 0 \end{pmatrix}$.

\item If $\asi(\E)=1$, and $\an(\E)$ has dimension $2$, then $\E$ is isomorphic to the evolution algebra with structure matrix $\tiny\begin{pmatrix}
0 & 0 & 0\cr 0 & 0 & 0\cr 0 & 0 & 1 \end{pmatrix}$.
\item If $\asi(\E)=\dim(\an(\E))=1$ then $\E\cong \Ad(B,\a)$, where $B$ is a two-dimensional evolution algebra with zero annihilator provided with a compatible symmetric bilinear form $\a$. The classification of evolution algebras in this class is given by the moduli set  \eqref{crepus}.
\end{enumerate}
\end{theorem}

\begin{proof}

Notice that the maximum value of $\asi(\E)$ is $3$. In this case $\E$ is nilpotent.
Since the classification of low dimensional nilpotent evolution algebras has been achieved up to dimension $5$ (see \cite{ElduqueLabra}), we will not pursue further this class of algebras. So we will focus on the cases $\asi(\E)\le 2$. Let us first consider $\asi(\E)=2$ and $\dim(\an(\E))=1$. 
Suppose that $\dim(\an^{(2)}(\E))=2$.
Take a generator $e_1$ of $\an(\E)$. We have 
$\an^{(2)}(\E)=\K e_1\oplus\K e_2$ for a suitable element $e_2\in\E$. We know that $e_2\E\subset\K e_1$.
At this point we know that $e_2^2=\l e_1$.
Choose an $a\in\E$ which is linearly independent with $e_1$ and $e_2$. So $\{e_1,e_2,a\}$ is a basis of $\E$. If $e_2 a=0$ then we have a natural basis $\{e_1,e_2,a\}$. On the contrary
$e_2 a=k e_1$ for a suitable $k\in\K^\times$.
Define $e_3:=x e_1+y e_2+z a$ (so that a wise choice of $x,y,z\in\K$ will give a natural basis 
$\{e_1,e_2,e_3\}$ of $\E$). To pick $x,y,z$, notice that we need 
\[0=e_2e_3=e_2(x e_1+y e_2+z a)=y\l e_1 +z k e_1,\]
and hence $y\l+z k=0$, so that $z=-y\l/k$, which proves that such a natural basis $\{e_1,e_2,e_3\}$ exists. We have $e_1^2=0$, $e_2^2=\l e_1$. As $\dim(\an^{(2)}(\E))=2$ then $e_3^2=\a e_1+\b e_2 + \delta e_3$ with $\delta \neq 0$. By Lemma~$\eqref{starving}$ we have the desired basis.

If $\a=0$ then $\{e_1,e_2,\b^{-1}e_3\}$ is a natural basis with structure matrix \[\tiny\begin{pmatrix}
0 & 1 & 0\cr 0 & 0 & 0\cr 0 & 0 & 1\end{pmatrix}.\]
If $\a\ne 0$ then $\{\a^2\b^{-4}e_1,\a\b^{-2}e_2,\b^{-1}e_3\}$ is a natural basis with structure matrix 
\[\tiny\begin{pmatrix}
0 & 1 & 0\cr 0 & 0 & 1\cr 0 & 0 & 1\end{pmatrix}.\]

Notice that the two algebras above are not isomorphic, since any isomophism between these two algebra preserves the annihilators and the set of nonzero idempotents.

Now suppose that $\dim(\an^{(2)}(\E))=3$. It is easy to check the structure matrix of this algebra is \[\tiny\begin{pmatrix}
0 & \a & \b \cr 0 & 0 & 0 \cr 0 & 0 & 0\end{pmatrix},\] where $\a$ and $\b$ are nonzero. After a change of basis we can take $\a=1$. Moreover, if we have two evolution algebras with structure matrices $\tiny\begin{pmatrix}
0 & 1 & \b \cr 0 & 0 & 0 \cr 0 & 0 & 0\end{pmatrix}$ and $\tiny\begin{pmatrix}
0 & 1 & \b' \cr 0 & 0 & 0 \cr 0 & 0 & 0\end{pmatrix}$ respectively, an easy though tedious computation reveals that both algebras are isomorphic if and only if $\b'=k^2\b$ for some $k\in\K^\times$. 
Thus a moduli set for this class of algebras is the one in \eqref{radiola}.
Clearly this algebra is not isomorphic to the two algebras with $\dim(\an^{(2)}(\E))=2$  considered above. 

Next we turn to the case  $\asi(\E)=2$ and $\dim(\an(\E))=2$. Since $\an(\E)=\langle e_1, e_2 \rangle$ there are two zero columns in the structure matrix of $\E$. Since  $\asi(\E)=2$, we have that $\E$ has a natural basis $\{e_1,e_2,e_3\}$ with  structure matrix
\[\tiny\begin{pmatrix}
0 & 0 & \alpha\cr 0 & 0 & \beta \cr 0 & 0 & 0\end{pmatrix},\] where $\a$ or $\b$ is nonzero. Notice that this algebra also has a natural basis given by $\{\a e_1+ \b e_2, e_2, e_3\}$, which has structure matrix
\[\tiny\begin{pmatrix}
0 & 0 & 1\cr 0 & 0 & 0 \cr 0 & 0 & 0\end{pmatrix}.\]

Let us now analyze the case $\asi(\E)=1$, in this case we have \[\an(\E)=\an\nolimits^{(2)}(\E)=\cdots . \] Then  $\dim(\an(\E))\in\{0,1,2\}$. By hypothesis $\dim(\an(\E))\neq 0$. The easiest case left is
the one in which $\dim(\an(\E))=2$, because we have a basis of the annihilator $\{e_1,e_2\}$ which can be completed to a natural basis $\{e_1,e_2,w\}$ of $\E$:
define $e_3:=xe_1+ye_2+zw$ where $z\ne 0$.
We have $w^2=a e_1+b e_2+c w$ for some 
$c\in\K^\times$. Observe that if $c=0$, then 
$w\E\subset\an(\E)$ so that $w\in\an^{(2)}(\E)=\K e_1+\K e_2$ a contradiction.
Then we compute $x,y,z$ so that $e_3^2$ gets simplified:
\[e_3^2=z^2(a e_1+b e_2+c w)=z^2 a e_1+z^2b e_2+z c(e_3-xe_1-ye_2)=\]
\[z(za-cx)e_1+z(zb-cy)e_2+z c e_3.\]
Since $c\ne 0$, we may take $z=c^{-1}$,
$x=zac^{-1}$ and $y=zbc^{-1}$ which imply $e_3^2=e_3$. Thus we have a natural basis 
$\{e_1,e_2,e_3\}$ with structure matrix 
$\tiny\begin{pmatrix}
0 & 0 & 0\cr 0 & 0 & 0\cr 0 & 0 & 1 \end{pmatrix}$.


Finally, the case $\asi(\E)=1$ with $\dim(\an(\E))=1$ follows from Lemma~\ref{chiringuito}. Observe also that $\an(\E/\an(\E))=0$ because $\asi(\E)=1$. A moduli set for this class of algebras is given in Remark \ref{pesado}.

\end{proof}

\section{Socle chain}\label{socchain}
In this section we will determine the socle in non-degenerate finite evolution algebras. We establish a new type of adjunction algebras and we analyse the conditions for these algebras to be isomorphic. These results will be useful for the classification in terms of the socle of non-degenerate three-dimensional evolution algebras. 

We start studying minimal ideals of evolution algebras.  These ideals are not necessarily simple algebras when considered as algebras on its own. Even if we assume that a minimal ideal has nonzero product, it is not a simple algebra as the following example shows.
\begin{example}\label{superejemplo}
\rm
Consider the three-dimensional evolution algebra with natural basis $\B=\{e_1, e_2, e_3\}$ and product $e_1^2=e_2+e_3$, $e_2^2=-e_3^2=e_1+e_2+e_3$. Let $I$ be the ideal generated by $e_1$, so $I=\sp(e_1,e_2+e_3)$. Note that $I^2 \neq 0$ and $I$ is not simple because $J=\sp(e_2+e_3)$ is a proper ideal of $I$. Moreover, $I$ is a minimal ideal of the evolution algebra.
\end{example}

Since minimal ideals will play a roll in our study, we next delimit the ground in which minimal ideals live. 
Furthermore, The next result gives the key for the effective computation of the socle of an evolution algebra.

\begin{proposition}\label{paleta}
Let $\E$ be  evolution algebra with natural basis $\B=\{e_i\}$. Let $I\triangleleft\E$ be minimal and $i\in\supp_\B(I)$ be such that $e_i^2\neq 0$. Then $I$ is generated (as an ideal) by $e_i^2$.
\end{proposition}
\begin{proof}
Let $i\in\supp_\B(I)$ as in the hypothesis. Then there exists a $x\in I$ such that $x=\lambda_i e_i+\theta$, where $\lambda_i\in \K^\times$ and $\theta$ is in the linear span  of the remaining basis elements. Multiplying $x$ by $e_i$ we obtain that  
$x e_i=\lambda_i e_i^2\in I$ and hence $e_i^2\in I$. Now, by the minimality of $I$, we have that $I=\esc{e_i^2}$.
\end{proof}

In the case of an associative ring $U$ (not necessarily unital), a (left) $U$-module $M$ is said to be irreducible (or simple according to other authors) when $UM\ne 0$ and the only submodules of $M$ are $0$ and $M$ itself (see for instance \cite[Definition 1, p. 4]{jac}). The (left) socle of   an $U$-module $M$ is defined as the sum of all irreducible submodules of $M$ (see \cite[Definition 1, p. 63]{jac}). Then, the socle of a ring $U$ is defined as the socle of the (left) $U$-module $U$. Consequently the socle of $U$ is the sum of all minimal left ideals $I$ such that $UI\ne 0$. For instance, if $U=\K$ is a field then its socle is $\K$ itself. However, if we consider the ring $U=\K$ (a field) endowed with the zero product, then its socle is $0$. The socle of a module over an associative ring is proved in \cite[Corollary 2, p. 61]{jac} to be a direct sum of irreducible $U$-submodules.

 If $A$ is any $\K$-algebra, we define the left multiplication algebra, $\mathfrak{M}$, as the subalgebra of 
$\text{End}_\K(A)$ generated by the left multiplication operators $L_a$ ($a\in A$). We also define the left multiplication algebra with unit, $\mathfrak{M_1}$, as the subalgebra of 
$\text{End}_\K(A)$ generated by the left multiplication operators $L_a$ ($a\in A$) and the identity map $A\to A$. 
Then $A$ is a left $\mathfrak{M}$-module, 
and the simple left $\mathfrak{M}$-modules of $A$ are those minimal left ideals $I$ of $A$ such that $AI\ne \{0\}$. So 
the (left) socle of $A$, denoted $\soc(A)$, can be defined as the socle of the $\mathfrak{M}$-module $A$. Applying the classical socle theory we have $\soc(A)=\oplus_\a J_\a$, where $\{J_\a\}$ is a certain collection of minimal left ideals of $A$, each one satisfying $AJ_\a\ne \{0\}$. If $A$ turns out to be commutative then $\soc(A)$ is a direct sum of minimal ideals of $A$ (not annihilated by $A$).
In the finite-dimensional case it is clear that the socle is always nonzero.  

We define the chain of socles of an evolution algebra as usual:
\[\soc(A)\subset\soc\nolimits^{(2)}(A)\subset\cdots\subset\soc\nolimits^{(n)}(A)\subset\soc\nolimits^{(n+1)}(A)\subset\cdots\]
where $\soc(A/\soc^{(n)}(A))=\soc^{(n+1)}(A)/\soc^{(n)}(A)$. This implies that each $\soc^{(n)}(A)$ is an ideal.

\begin{definition} \rm \label{ssi} In the previous setting, suppose that there is an $n\in \N^*$ such that  $\soc^{(n)}(A)=\soc^{(n+k)}(A)$ for each $k>0$. We then define the {\em socle stabilizing index}, denoted $\ssi(A)$, as the least natural $n$ such that $\soc^{(n)}(A)=\soc^{(n+k)}(A)$ for any $k>0$. 
\end{definition}

Proposition \ref{paleta} gives a procedure for the effective computation of the socle of a finite-dimensional non-degenerate evolution algebra, as we illustrate in the following example.

\begin{example}\label{minimales}
\rm
 Consider, for instance, the $4$-dimensional algebra $\E$ whose structure matrix relative to the natural basis $\{e_i\}_{i=1}^4$ is 

\[\left(\begin{matrix}
1 & 1 & 0 & 0 \\
0 & 1 & 0 & 0 \\
0 & 0 & -1 & 1 \\
0 & 0 & -1 & 1
\end{matrix}\right).\]

We apply Proposition \ref{paleta} to compute the socle of $A$.
Then $\esc{e_1^2}=\bK e_1$ and hence this ideal is minimal. Since 
$\esc{e_2^2}=\sp(\{e_1,e_2\})$ this ideal is not minimal. Finally 
$\esc{e_3^2}=\esc{e_4^2}=\sp(\{e_3+e_4\})$, and this ideal is also minimal. So we have only two minimal ideals $\esc{e_1^2}$ and
$\esc{e_3^2}$. Both ideals are not annihilated by $\E$, hence $\soc(\E)=\esc{e_1^2}\oplus\esc{e_3^2}$.
\end{example}

Recall that if $A$ is an evolution algebra which decomposes as a direct sum of (possibly infinitely many) ideals
$A=\oplus_{\a\in\Lambda}J_\a$, then each ideal $J_\a$ is an evolution algebra (see \cite[Lemma 5.2.]{CSV1}).

\subsection{Adjunction of  type two.} \label{taza}
Let $B$ be an evolution algebra over $\K$. In this subsection our goal is to construct a commutative algebra $A$ containing $B$ as a minimal ideal. Furthermore, we want $A/B$ to be a one-dimensional algebra.

As a vector space $A=B\times\K$ and the multiplication of $A$ must 
be $(b,\l)(b',\l'):=(bb'+\a(b,\l')+\a(b',\l)+\Phi(\l,\l'),\l \l'k_0)$,
where $k_0 \in \K$; $\a\colon B\times\K\to B$ and $\Phi\colon\K\times\K\to B$ are bilinear maps. Then there is a linear map $\varphi\colon B\to B$ such that $\varphi(b)=\alpha(b,1)$ for each $b \in B$ (i.e. $\lambda\varphi(b)=\alpha(b,\l)$ with $\l \in \K$). Since we want our construction to give a zero annihilator algebra, there must exist an element $0\ne b_0\in B$ such that $\Phi(\l,\l')=\l\l'b_0$. Thus, we can rewrite the product in $A$ in the form:
\begin{equation}\label{loque}
(b,\l)(b',\l'):=(bb'+\l'\varphi(b)+\l\varphi(b')+\l\l'b_0,\l \l'k_0),
\end{equation}
with $\varphi\colon B\to B$ linear and $b_0\in B\setminus\{0\}$. Then $B\times\{0\}$ is an ideal of $A$  isomorphic to $B$ so we will identify them. 
Under these circumstances we have a short exact sequence
\[B\buildrel{i}\over{\hookrightarrow}  A\buildrel{\pi}\over\twoheadrightarrow \K,\]
where $\K$ is endowed with the product $\l\ast \l'=k_0\l\l'$, $i$ is the canonical injection $i(b):=(b,0)$, and $\pi$ is the canonical projection.
\begin{definition}\label{adjdos} \rm
Let $B$ be an evolution algebra over $\K$ and $\varphi\colon B\to B$ a linear map. Consider moreover $b_0\in B$ and $k_0\in \K$. We define  $\Au{(B,\varphi,b_0,k_0)}$ as the algebra $B \times \K$ defined above with the product \eqref{loque}. Occasionally, to simplify, we write  $\Au{(B,\varphi,b_0)}$ instead of  $\Au{(B,\varphi,b_0,0)}$.
\end{definition}

Since we want $B\times\{0\}$ to be a minimal ideal of $A=\Au{(B,\varphi,b_0,k_0)}$, we must require the additional condition that 
the unique ideals of $B$ which are $\varphi$-invariant are $0$ and $B$ itself. Indeed we have

\begin{lemma}\label{naranja} Assume that $A=\Au(B,\varphi,b_0,k_0)$ has zero annihilator. 
Then  $B\times\{0\}$ is a minimal ideal of $A$ if and only if the unique ideals of $B$ which are $\varphi$-invariant are $0$ and $B$ itself. Furthermore, if we assume that $B\times\{0\}$ is minimal, then $\soc(A)=B\times \{0\}$ or $A=\soc(A)$.
\end{lemma}
\begin{proof}
 Assume first that
$B\times\{0\}$ is a minimal ideal of $A$ and take any ideal $0\ne I\triangleleft B$ satisfying $\varphi(I)\subset I$.
Then $I\times\{0\}$ is a nonzero ideal of $A$ contained in $B\times\{0\}$ and  hence $I=B$.
Conversely assume that $B$ has no $\varphi$-invariant ideals other than $0$ and $B$. Take an ideal $J$ of $A$ contained in $B$. Then $J=I\times\{0\}$, where $I\triangleleft B$  and $\varphi(I)\subset I$. Hence $I=0$ or $I=B$ implying that $J=0$ or $J=B\times\{0\}$. Let us prove the second assertion. Assume that $B\times\{0\}$ is a minimal ideal of $A$ and $J$ is any other minimal ideal of $A$. If $J\cap (B\times\{0\})=0$, then $\soc(A)=J\oplus (B\times\{0\})=A$. So if $A$ does not coincide with its socle then we have $0\ne J\cap (B\times\{0\})\subset B\times\{0\}$. By minimality of $B\times\{0\}$ we have 
$B\times\{0\}\subset J$ and so, by minimality of $J$, we have $B\times\{0\}=J$. Consequently $\soc(A)=B\times\{0\}$.
\end{proof}

\subsection{Isomorphisms between adjunction algebras of type two.}

We next explicit the isomorphism conditions among non-semisimple algebras $\Au(B,\varphi,b_0,k_0)$ with zero annihilator and such that 
$\soc(\Au(B,\varphi,b_0,k_0))=B\times\{0\}$. In order to do this, we will consider a $\K$-algebra homomorphism
 $\theta\colon \Au{(B,\varphi,b_0,k_0)}\cong \Au{(B',\varphi',b_0',k_0')}$ mapping $B\times\{0\}$ to $B'\times\{0\}$ (a condition which is automatic if $\theta$ is an isomorphism, because  $\theta$ would map the socle of the first algebra into the socle of the second one, take also into account Lemma \ref{naranja}).
 Firstly, for any $b\in B$ and $\l\in\K$, we can write $\theta(b,\l)=(\theta_1(b)+\theta_2(\l),\theta_3(\l))$, where $\theta_1\colon B\to B'$, 
$\theta_2\colon \K  \to B'$, and $\theta_3\colon \K\to\K$ are linear maps.  It is easy to check that $\theta_1$ is a homomorphism of $\K$-algebras and satisfy
\begin{equation}
\begin{cases}\label{paren}
\theta_1\varphi = L_{\theta_2(1)}\theta_1+\theta_3(1)\varphi'\theta_1,\cr 
k_0\theta_3(1)=k_0'\theta_3(1)^2,\cr
\theta_1(b_0)+k_0\theta_2(1)=\theta_2(1)^2+2\theta_3(1)\varphi'(\theta_2(1))+\theta_3(1)^2b_0'.\end{cases}
\end{equation}
If $\theta$ turns out to be an isomorphism, then $\theta_1:B\to B'$ is an isomorphism. Hence we will consider the isomorphism problem in the mentioned class of algebras $\Au(B,\varphi,b_0,k_0)$, but with $B$ fixed.
As a consequence of equations \eqref{paren}, the algebras $\Au{(B,\varphi,b_0,0)}$ and $\Au{(B,\varphi',b_0',0)}$
are isomorphic if, and only if, $\theta_1\varphi = L_{\theta_2(1)}\theta_1+\theta_3(1)\varphi'\theta_1$ and $\theta_1(b_0)=\theta_2(1)^2+2\theta_3(1)\varphi'(\theta_2(1))+\theta_3(1)^2b_0'$.
No algebra $\Au{(B,\varphi,b_0,0)}$ is isomorphic to $\Au{(B,\varphi',b_0',k_0')}$ with $k_0'\ne 0$. Furthermore, two algebras $\Au{(B,\varphi,b_0,k_0)}$ and $\Au{(B,\varphi',b_0',k_0')}$ are isomorphic if and only if
\begin{equation}\label{parenu}
\begin{cases}
\theta_1\varphi\theta_1^{-1} = L_{\theta_2(1)}+\theta_3(1)\varphi',\cr 
k_0=k_0'\theta_3(1),\cr
\theta_1(b_0)+k_0\theta_2(1)=\theta_2(1)^2+2\theta_3(1)\varphi'(\theta_2(1))+\theta_3(1)^2b_0'.\end{cases}
\end{equation}
As a corollary, any algebra $\Au{(B,\varphi,b_0,k_0)}$, with $k_0\ne 0$, is isomorphic to a suitable algebra $\Au{(B,\varphi',b_0',1)}$. Summarizing: the isomorphism conditions among algebras of type $\Au(B,\varphi,b_0,k_0)$ is reduced to the following two questions:
\begin{enumerate}
    \item When is there an isomorphism $\Au{(B,\varphi,b_0,0)}\cong \Au{(B,\varphi',b_0',0)}$ ?
\item When is there an isomorphism $\Au{(B,\varphi,b_0,1)}\cong \Au{(B,\varphi',b_0',1)}$ ?
\end{enumerate}
Concerning the second question, by $\eqref{parenu}$ we would have $\theta_3(1)=1$. If we take the particular solution $\theta_2=0$, then we get $\theta_1\varphi\theta_1^{-1}=\varphi'$ and $\theta_1(b_0)=b_0'$. Thus 
\[\forall\theta\in\aut(B), \hbox{ we have }\Au(B,\varphi,b_0,1)\cong\Au(B,\theta\varphi\theta^{-1},\theta(b_0),1).\]

In order to address the general answer to the above isomorphism questions, notice that (\ref{parenu}) yields the following result:

\begin{proposition}\label{pollo}
Let $B$ be a $\K$-algebra, and $\varphi,\varphi'\colon B\to B$ be linear maps with no non-trivial invariant ideals. Assume that $\Au(B,\varphi,b_0,0)$ is not semisimple and has zero annihilator. Then 
$\Au(B,\varphi,b_0,0)\cong \Au(B,\varphi',b_0',0)$ if and only if there is an automorphism 
$\theta\in\aut(B)$, and elements
$k\in\K^\times$, $b_1\in B$, such that:
\begin{equation}
\begin{cases}
\varphi'=k^{-1}(\theta\varphi\theta^{-1}-L_{b_1}),\cr 
b_0'=k^{-2}\left(\theta(b_0)-b_1^2-2k\varphi'(b_1)\right).
\end{cases}
\end{equation}
In particular, we have $\Au(B,\varphi,b_0,0)\cong \Au(B,\theta\varphi\theta^{-1},\theta(b_0),0)$ for any $\theta\in\aut(B)$.
\end{proposition}

To get a classifying moduli set for the algebras of type $\Au(B,\varphi,b_0,0)$, let $\G:=\aut(B)\times B\times\K^\times$ be the group with product defined by
\[(\theta,b_1,k)(\theta',b_1',k'):=(\theta\theta', \theta(b_1')+k' b_1,k k').\]
Consider the set $\endo_{\K}{(B)}\times B$ and the action $\G\times [\endo_{\K}{(B)}\times B]\to [\endo_\K{(B)}\times B]$
given by $(\theta,b,k)\cdot(\varphi,b_0)=(\varphi',b_0')$, where 
\[\begin{cases}
\varphi'=k^{-1}(\theta\varphi\theta^{-1}-L_{b}),\cr 
b_0'=k^{-2}\left(\theta(b_0)-b^2-2k\varphi'(b)\right).
\end{cases}\]

Then Proposition~\ref{pollo} can be re-stated, in terms of the above action, by claiming that the isomorphism classes of algebras $\Au(B,\varphi,b_0,0)$ are in one-to-one correspondence with the elements of the set $(\endo_\K(B)\times B)/\G$ (or in other words, the algebras of type $\Au(B,\varphi,b_0,0)$
are classified by the moduli set $(\G,\endo_\K(B)\times B)$).

We now focus on the isomorphism question for algebras of type $\Au(B,\varphi,b_0,1)$. As before, we use equations in (\ref{parenu}) to obtain the following result.

\begin{proposition}\label{pollo1}
Let $B$ be a $\K$-algebra. Then
$\Au(B,\varphi,b_0,1)\cong \Au(B,\varphi',b_0',1)$ if and only if there is an automorphism 
$\theta\in\aut(B)$, and an element $b_1\in B$, such that: 
\[\begin{cases}\varphi'=\theta\varphi\theta^{-1}-L_{b_1},\cr 
b_0'=\theta_1(b_0)+b_1+b_1^2-2\theta\varphi\theta^{-1}(b_1).
\end{cases}\]
\end{proposition}

To get a classifying moduli set for the algebras of type $\Au(B,\varphi,b_0,1)$, let $\G:=\aut(B)\times B$ be the group with product defined by
\[(\theta,b_1)(\theta',b_1'):=(\theta\theta', \theta(b_1')+b_1).\]
Consider the set $\endo_{\K}{(B)}\times B$ and the action $\G\times [\endo_{\K}{(B)}\times B]\to [\endo_\K{(B)}\times B]$
given by $(\theta,b)\cdot(\varphi,b_0)=(\varphi',b_0')$, where 
\begin{equation}\label{poy}
\begin{cases}\varphi'=\theta\varphi\theta^{-1}-L_{b},\cr 
b_0'=\theta_1(b_0)+b+b^2-2\theta\varphi\theta^{-1}(b).
\end{cases}
\end{equation}
Then Proposition~\ref{pollo1} can be re-stated in terms of the above action by claiming that 
the isomorphism classes of algebras $\Au(B,\varphi,b_0,1)$ (for a fixed $B$) are in one-to-one correspondence with the elements of the set $(\endo_\K(B)\times B)/\G$.

\begin{example}\rm 
Let $B$ be the evolution $\R$-algebra with natural basis $\B=\{e_1,e_2\}$ and multiplication $e_1^2=e_2$ and $e_2^2=e_1$. Then $\G:=\aut(B)=\{1,\theta\}\cong\Z_2$, where $\theta$ swaps $e_1$ and $e_2$. We want to classify the algebras $\Au(B,\varphi,b_0,1)$. Observe that for any $b\in B$ we have $b=\a e_1+\b e_2$ with $\a,\b\in\R$, so that the matrix of $L_b$ relative to $\B$, by columns, is $\tiny\begin{pmatrix}0 & \b\cr\a & 0\end{pmatrix}$. On the other hand, for any $(\varphi,b_0)$ there is some $b$ such that the matrix of $\varphi'$ relative to $\B$ is diagonal (see equations in \eqref{poy}). So the action of the group allow us to find a $(\varphi',b_0')$ in the orbit of each $(\varphi,b_0)$, with $\varphi'$ diagonal relative to the basis $\B$. Consequently we focus in the problem of classifying the pairs $(x,y)$ and $(z,t)$, where $\varphi$ is represented by the diagonal matrix $\hbox{\rm diag}(x,y)$ relative to $\B$, and $b_0$ has the coordinates $(z,t)\ne (0,0)$ relative to the same basis. In these terms, the action is equivalent to

\[\G\times [\R^2\times(\R^2\setminus\{0\})]\to \R^2\times(\R^2\setminus\{0\}),\]
where the element $\theta\in\aut(B)$ acts in the form $\theta\cdot (x,y,z,t)=(y,x,t,z)$. It can be proved that $[\R^2\times(\R^2\setminus\{0\})]/\G$ is in one-to-one correspondence with the subset of $\R^4$ given by 
\[\{(x,y,z,t)\in\R^4\colon z<t\}\sqcup\{(x,y,z,z)\in \R^4\colon x\le y,z\ne 0\}.\]

\end{example}

Now, in order to establish a relation between $\ann(B)$  and $\Au{(B,\varphi,b_0)}$ we prove the following proposition.

\begin{proposition}\label{brownie} Let $A=\Au{(B,\varphi,b_0)}$. We have
\begin{enumerate}
    \item \label{tarta}$\Ker(\varphi)=\{b\in B\colon (b,0)(0,1)=0\}=\an_{B\times \{0\}}(\{0\}\times\K)$.
    \item\label{galleta} $\an(A)\cap (B\times \{0\})=(\an(B)\cap\Ker(\varphi))\times \{0\}$.
    \item \label{vainilla} If $(\an(B)\cap\Ker(\varphi))\times \{0\}=\{0\}$, then $\dim(\an(A))\in\{0,1\}$.
    \item \label{sirope} If there exists $b\in B$ such that $b^2=b_0$ and $\varphi=-L_b$, then $(b,1)\in \ann(A)$.
    \item \label{bizcocho}If  $(\an(B)\cap\Ker(\varphi))\times \{0\}=\{0\}$ and $\an(A)\ne \{0\}$, then there is a $b\in B$ such that $b^2=b_0$ and $\varphi=-L_b$. In this case 
    $\an(A)=\K(b,1)$.
    \item \label{caramelo} If $(\an(B)\cap\Ker(\varphi))\times \{0\}\ne 0$, $\an(A)\ne \{0\}$, there is a $b\in B$ such that $b^2=b_0$, and $\varphi=-L_b$, then   
    $\an(A)=\K(b,1)\oplus\left(\an(B)\cap\Ker(\varphi)\times\{0\}\right)$.
    \item \label{mermelada} If $(\an(B)\cap\Ker(\varphi))\times \{0\}\ne 0$ and $\an(A)\ne \{0\}$, but there is no $b\in B$ such that $b^2=b_0$ and $\varphi=-L_b$, then   
    $\an(A)=\an(B)\cap\Ker(\varphi)\times\{0\}$. 
\end{enumerate}
\end{proposition}
\begin{proof}
To prove item \eqref{tarta} notice that if $b\in\Ker(\varphi)$ then we have $(b,0)(0,\l)=(\l\varphi(b),0)=(0,0)$ and hence $(b,0)\in\an_{B\times \{0\}}(\{0\}\times\K)$. Conversely if 
$(b,0)(\{0\}\times\K)=\{0\}$ we have $\varphi(b)=0$. 

For the assertion \eqref{galleta} take $(b,0)\in\an(A)$. Then 
$(b,0)(x,\l)=0$ for any $x\in B$ and $\l\in\K$. Thus $bx+\l\varphi(b)=0$, which implies that $b\in\an(B)$ (taking $\l=0$) and  $b\in\Ker(\varphi)$ (taking $x=0$). Therefore $(b,0)\in(\an(B)\cap\Ker(\varphi))\times \{0\}$. Now if 
$(b,0)\in (\an(B)\cap\Ker(\varphi))\times \{0\}$ we have 
$bB=\{0\}$ and $\varphi(b)=0$. So 
$(b,0)(x,\l)=(bx+\l\varphi(b),0)=(0,0)$, what amounts to say that $(b,0)\in\an(A)\cap(B\times \{0\})$. 

Next we prove the statement \eqref{vainilla}. We assume $(\an(B)\cap\Ker(\varphi))\times \{0\}=\{0\}$. If $\an(A)\ne \{0\} $ take an arbitrary nonzero $(b,\l)\in\an(A)$. Observe that $\l\ne 0$, or else $(b,\l)\in\an(A)\cap(B\times \{0\})=\{0\}$ by the previously proved item. Now, for any other nonzero element $(b',\l')\in\an(A)$ we also have $\l'\ne 0$, so that there is a nonzero scalar $k$ with $\l'=k\l$. This implies that $k(b,\l)-(b',\l')=(kb-b',0)\in\an(A)\cap(B\times \{0\})=\{0\}$. So we have proved that any nonzero element in $\an(A)$ is a multiple of $(b,\l)$. Whence $\dim(\an(A))=1$. 

For statement \eqref{sirope}, consider $(x,\l) \in A$. Then $(b,1)(x,\l)=(bx+\l\varphi(b)+\varphi(x)+\l b_0,0)=(0,0)$, and so $(b,1)\in \ann(A)$. 

To prove the assertion \eqref{bizcocho} notice that we already know that $\an(A)$ has dimension $1$. Take a generator of the form $(b,1)$ of $\an(A)$ (this is unique). Then $0=(b,1)(x,\l)=(bx+\l\varphi(b)+\varphi(x)+\l b_0,0)$ for any $x\in B$ and $\l\in\K$. This implies that $bx+\varphi(x)=0$ for any $b \in B$ and $\varphi(b)=-b_0$. Then $\varphi=-L_b$, so that $b^2=b_0$.  

Next we prove item  \eqref{caramelo}. Observe that 
$\K(b,1)\subset\an(A)$ and, by what was proved in item \eqref{galleta}, $\left(\an(B)\cap\Ker(\varphi)\times\{0\}\right)\subset\an(A)$. So we have 
$ \K(b,1)\oplus\left(\an(B)\cap\Ker(\varphi)\times\{0\}\right)\subset\an(A)$.
Finally, if $(b',\l)\in\an(A)$ and $\l=0$, we apply item (\ref{galleta}) and obtain 
$(b',\l)\in (\an(B)\cap\Ker(\varphi))\times\{0\}$. If $\l\ne 0$ we have 
$(b',\l)=\l(b'',1)$, where $(b'',1)\in\an(A)$. Since $(b,1)\in\an(A)$, we have 
$(b''-b,0)\in\an(A)\cap(B\times\{0\})$, and by item~\eqref{galleta} 
$(b''-b,0)\in(\an(B)\cap\Ker(\varphi))\times\{0\}$. Thus 
\[(b',\l)=\l(b'',1)=\l(b+b''-b,1)=\l(b,1)+\l (b''-b,0)\in\K(b,1)\oplus\left(\an(B)\cap\Ker(\varphi)\times\{0\}\right).\]

To prove the final item (\ref{mermelada}) notice that, on one hand, it is straightforward to check that $(\an(B)\cap\Ker(\varphi))\times\{0\}\subset\an(A)$. For the converse inclusion, take $(b,\l)\in\an(A)$. We prove that necessarily $\l=0$. Otherwise there is an element of the form $(b,1)\in\an(A)$. But then, for any $x\in B$ we have $0=(b,1)(x,0)=(bx+\varphi(x),0)$, giving $\varphi=-L_b$. Also $0=(b,1)(0,1)=(\varphi(b)+b_0,0)$, giving $\varphi(b)=-b_0$. Since $\varphi=-L_b$ we have $b^2=b_0$. Since such a $b$ does not exist, we conclude that $\l=0$ and $\an(A)\subset (\an(B)\cap\Ker(\varphi))\times\{0\}$.
 \end{proof}

\begin{corollary}\label{flan}
Let $A=\Au{(B,\varphi,b_0)}$. Then the following assertions hold:
\begin{enumerate}
\item If $\an(B)=0$ and there is no $b\in B$ with $b^2=b_0$ and $\varphi=-L_b$, then $\an(A)=0$.
\item If $\an(A)=0$, then there is no $b\in B$ such that $b^2=b_0$ and $\varphi=-L_b$.
\end{enumerate}
\end{corollary}
\begin{proof}

The first item of the corollary follows from items \eqref{vainilla} and \eqref{bizcocho} in Proposition \ref{brownie}. The second item follows from item \eqref{sirope}.

\end{proof}

As it happens with $\Ad$, the adjunction $\Au{(B,\varphi,b_0,k_0)}$ is also an evolution algebra (for a suitable $\varphi)$, as we show in the next proposition.

\begin{proposition} If $B$ is an evolution algebra, then
$A=\Au{(B,\varphi,b_0,k_0)}$ is an evolution algebra for a suitable $\varphi$.
\end{proposition}
\begin{proof}
Take a natural basis $\{u_i\}$ of $B$ and $x$ any nonzero element. Define $\varphi\colon B\to B$ by $\varphi(u_i)=-xu_i$ for all $i$. Next we prove that the collection $\B:=\{(u_i,0)\}\cup\{(x,1)\}$ is a basis of $A$. First notice that $\B$ is a generator set. Indeed, since $x=\sum \l_iu_i$ for suitable scalars $\l_i \in \K$, we can write $(0,1)=-\sum \l_i(u_i,0)+(x,1)$. Now we will prove that $\B$ is linearly independent. If $\sum \l_i(u_i,0)+\mu(x,1)=(0,0)$ then 
$\mu=0$, implying that $\l_i=0$ for all $i$. Thus $\B$ is a basis of $A$ and 
$(u_i,0)(x,1)=(u_ix+\varphi(u_i),0)=0$ by the definition of $\varphi$. Since $B$ is a subalgebra of $A$ we conclude that $A$ is an evolution algebra.
\end{proof}

\section{Classification in terms of the socle}\label{classoc}
The degenerate evolution  algebras of dimension $3$ have been classified in Theorem \ref{tornillo}.
So, our main goal in this section is the classification of non-degenerate $3$-dimensional evolution algebras in terms of the socle.
These have nonzero socle, therefore we must start considering cases in terms of the dimension of the socle. 

\rm As a first observation regarding the dimension of the socle, notice that if a commutative algebra $A\ne 0$ has zero annihilator and $\dim(\soc(A))>1$ then $\dim(A^2)>1$. Indeed, if $A^2=\K a$ for some nonzero $a$, and we take a nonzero ideal $I$ of $A$, then $IA\subset \K a$ and $IA\ne 0$.
Thus $IA=\K a\subset I$. So $\K a$ is the unique minimal ideal of $A$, which implies that $\soc(A)=\K a$ has dimension $1$.

Before we present our classification result in terms of the socle, wee need two auxiliary results, which we prove below.
Recall that for a inner product $\esc{\cdot , \cdot}\colon V\times V\to \K$
(i.e. symmetric bilinear form) on a $\K$-vector space $V$, the radical of $\esc{\cdot , \cdot}$ is defined as the
subspace $\rad(\esc{\cdot , \cdot}):=\{v\in V\colon \esc{v,V}=0\}$.

\begin{lemma}\label{grajo}
Let $A$ be an evolution $\K$-algebra and $0\ne u\in A$. Fix a subspace $B$ such that $A=\K u\oplus B$ and 
consider the linear map $p\colon A\to \K$ such that $a=p(a)u+b$ according to the decomposition $A=\K u\oplus B$. Define an inner product $\pi\colon A\times A\to \K$ by $\pi(x,y):=p(xy)$ for any $x,y\in A$. Consider now any natural basis $\B=\{b_i\}_{i\in\Lambda}$ of $A$. Then \[\rad(\pi)=\sp{\{b_i\colon \pi(b_i,b_i)=0\}}.\]
In particular, if $\rad(\pi)$ is one-dimensional, any generator of this radical is (up to nonzero scalar multiples) contained in every natural basis of $A$.
\end{lemma}
\begin{proof}
Take an arbitrary $x\in\rad(\pi)$ and write $x=\sum_i x_i b_i$ with $x_i\in \K$. Then for any $j$ we have
$0=\pi(x,b_j)=\sum_i x_i\pi(b_i,b_j)=x_j\pi(b_j,b_j)$. So, if $x_j\ne 0$ then $\pi(b_j,b_j)=0$ and hence
$x\in\sp{\{b_i\colon \pi(b_i,b_i)=0\}}$. The other inclusion is clear.
\end{proof}

\begin{lemma}\label{cabrao} Let $A$ be a three-dimensional evolution algebra with a two-dimensional ideal $I$. Then, one of the following mutually excluding possibilities holds.
\begin{enumerate}
    \item \label{uno} $I$ is the linear span of two of the elements in a natural basis of $A$.
    \item \label{dos} There is no natural basis such that $I$ is generated by two of its elements, but
there exists a natural basis $\B=\{e_1,e_2,e_3\}$ of $A$ such that $I$ is the linear span of $e_i$ and $e_j+e_k$, with $i,j,k$ different. In this case $I$ is an evolution ideal.
     \item \label{tres} The ideal $I$ does not verify any of the cases \eqref{uno} and \eqref{dos}, but
     there exists a natural basis $\B=\{e_1,e_2,e_3\}$ of $A$ such that $I$ is the linear span of $e_i+e_j$ and $e_j+e_k$, with $i,j,k$ different. If $A$ is perfect, then $I$ is not an evolution ideal. 
\end{enumerate}
\end{lemma}
\begin{proof}
Let $\{e_i\}_{i=1}^3$ be a natural basis of $A$ with structure matrix   $M=(\w_{ij})$. Assume that $u_1=xe_1+ye_2+ze_3$, $u_2=x'e_1+y'e_2+z'e_3$ is a basis of the ideal $I$.
For $i=1,2,3$ denote \[\tiny\Delta_i:=\left|\begin{matrix}x & y & z\cr x' & y' & z'\cr \w_{1i} & \w_{2i} & \w_{3i} 
\end{matrix}\right|.\]
Since $e_1u_1=x\w_{11}e_1+x\w_{21}e_2+x\w_{31} e_3\in I$ we have that 
\[\hbox{rank}\begin{pmatrix}x & y & z\cr x' & y' & z'\cr x\w_{11} & x\w_{21} & x\w_{31} 
\end{pmatrix}=2,\hbox{ which implies } x\Delta_1=0.\]
Also from $e_1u_2\in I$ we deduce that $x'\Delta_1=0$. Proceeding in this way we get, not only that
$0=x\Delta_1=x'\Delta_1$, but also that $0=y\Delta_2=y'\Delta_2$ and $0=z\Delta_3=z'\Delta_3$.
If some $\Delta_i\ne 0$ then $I$ is the linear span of two elements of the basis $\{e_i\}$ of $A$. So we assume in the sequel that $\Delta_i=0$ for any $i=1,2,3$.
The row reduced echelon form of $\begin{pmatrix}x & y & z\cr x' & y' & z'\end{pmatrix}$ is one of the following:
\[\begin{pmatrix}1 & 0 & \alpha \cr 0 & 1 & \a'\end{pmatrix}, \begin{pmatrix}1 & 0 & \a\cr 0 & 0 & 1\end{pmatrix}, \begin{pmatrix}0 & 1 & \a\cr 0 & 0 & 1\end{pmatrix},\quad (\a,\a'\in\K).\]
So we can find a basis $\{v_1,v_2\}$ of $I$ such that one of the following holds:
\begin{enumerate}[(a)]
    \item \label{a} $v_1=e_1+\a e_3$, $v_2=e_2+\a' e_3$.
    \item \label{b} $v_1=e_1+\a e_3$, $v_2=e_3$.
    \item \label{c} $v_1=e_2+\a e_3$, $v_2=e_3$.
\end{enumerate}
In cases \eqref{b} and \eqref{c} we have a basis of $I$ consisting of two vectors of the natural basis. So we focus on the case \eqref{a}.
If $\a=\a'=0$, then again there is a basis of $I$ consisting of elements of the natural basis. Next we consider the following possibilities:
\begin{itemize}
    \item[\rm i)] If $\a =0$ and $\a' \ne 0$, or  $\a' =0$ and $\a \ne 0$, then we can construct a new  natural basis of $A$ such that $I$ is in the possibility \eqref{dos}.
\item[\rm ii)]If $\a \ne 0$ and $\a' \ne 0$, then we can construct a new  natural basis of $A$ such that $I$ is in the possibility \eqref{tres}.
\end{itemize}
Finally let us prove that if $A^2=A$ and $I$ is generated by $e_i+e_j$ and $e_j+e_k$, with $i, j,  k$ different, then $I$ is not an evolution algebra: 
assume on the contrary that $u_1=x(e_i+e_j)+y(e_j+e_k)$ and $u_2=x'(e_i+e_j)+y'(e_j+e_k)$ is a natural basis of $I$. Then
$0=xx'(e_i^2+e_j^2)+yy'(e_j^2+e_k^2)+(xy'+x'y)e_j^2=xx'e_i^2+yy'e_k^2+(xx'+yy'+xy'+yx')e_j^2$ and hence $xx'=yy'=xy'+yx'=0$, which is inconsistent with $xy'-yx'\ne 0$.
\end{proof}

\subsection{Socle of dimension three}

Under this hypothesis the algebra is either simple or a direct sum of simple evolution algebras of dimension $\le 2$. In fact the simple algebras 
provide moduli sets depending on up to $6$ parameters,  which is one of the most densely populated collection of isomorphism classes of algebras.

\begin{proposition}\label{threedim}
If $A$ is a $3$-dimensional evolution $\K$-algebra with zero annihilator, whose socle is $A$ itself, then
$A$ is (isomorphic to) one of the following:
\begin{enumerate}
    \item A simple evolution algebra.
    \item $A=B\oplus \K$, with $B$ a simple evolution algebra of dimension $2$, and the product in $A$ given by \[(b,\lambda)(b',\lambda')=(bb',\lambda\lambda').\]
    \item $\K\oplus \K\oplus \K$ with componentwise product.
\end{enumerate}
\end{proposition}
\begin{proof}
 If the socle consists of a unique minimal ideal then
it is a simple algebra. Otherwise we have a direct sum of two or three minimal ideals.
If $A=\soc(A)=I\oplus J$, with $\dim(I)=2$ and $\dim(J)=1$, then both ideals are evolution algebras (see \cite[Lemma 5.2.]{CSV1}) with zero annihilator.
Hence $J=\K e$, where $e$ is an idempotent,
and $I^2\ne 0$ since on the contrary $IA=0$. Thus $I$ is a simple evolution algebra, because any ideal of $I$ is an ideal of $A$ (given that $I$ is a summand of $A$). Finally, if the socle has three components then $A\cong \K^3$ with componentwise operations.
\end{proof}

\subsection{Socle of dimension two}
In this section we will study the case $\dim(\soc(A))=2$. 
We start this subsection describing a number of moduli sets which will appear in the classification Theorem~\ref{partwotwo}.

\begin{definition}\label{grfghgth}\rm
Recall that $(\K^\times)^{\esc{2}}:=\{k^2\colon k\in\K^\times\}$ is the group of nonzero squares of $\K$. 
\begin{enumerate}
    \item \label{chicas} Consider the group $\Z_2:=\{0,1\}$ and 
 the product group $(\K^\times)^{\esc{2}}\times\Z_2$. This group acts on the set $\M_2(\K)$ of $2\times 2$ matrices with coefficients in $\K$. The action is defined by
 \[
 \begin{array}{cc}
      &  [(\K^\times)^{\esc{2}}\times\Z_2]\times\M_2(\K)\to\M_2(\K) \\
      & (\lambda,i)\cdot M:=\lambda E_{12}^i ME_{12}^i,
 \end{array}
\]
\noindent where $E_{12}=\tiny\begin{pmatrix}0 & 1\cr 1 &0\end{pmatrix}$, $E_{12}^0=\id_{2}$ and $E_{12}^1=E_{12}$, see Notation~\ref{domingo}.

\item \label{chico} Consider  the group $(\K^\times)^{\esc{2}}\times \K^{\times}$. This group acts on the set $\M_2(\K)$ of $2\times 2$ matrices with coefficients in $\K$. The action is defined in the following way
\[
 \begin{array}{cc}
      &  [(\K^\times)^{\esc{2}}\times\K^\times]\times\M_2(\K)\to\M_2(\K) \\
      & (\alpha,\beta)\cdot M:=\tiny\begin{pmatrix}\a & 0\cr 0 &\b\end{pmatrix}M\tiny\begin{pmatrix}\a^{-2} & 0\cr 0 &\b^{-2}\end{pmatrix}.
 \end{array}
\]

\item \label{nino} The product group $\K^\times\times\Z_2$ acts on the set $\M_3(\K)$ of $3\times 3$ matrices with coefficients in $\K$.  The action is defined in the following way
 \[
 \begin{array}{cc}
      &  (\K^\times\times\Z_2)\times\M_3(\K)\to\M_3(\K) \\
      & (\lambda,i)\cdot \begin{pmatrix} M & v \cr x & \w_{33} \end{pmatrix}:=  \begin{pmatrix} \lambda^2 E_{12}^i ME_{12}^i & v \cr x & \lambda \w_{33} \end{pmatrix} 
 \end{array}.
\]

\item \label{nina} The product group $\K^\times\times\K^\times$ acts on the set $\M_3(\K)$ of $3\times 3$ matrices with coefficients in $\K$. The action is given by
 \[
 \begin{array}{cc}
      &  (\K^\times\times\K^\times)\times\M_3(\K)\to\M_3(\K) \\
      & (\a,\b)\cdot \begin{pmatrix} M & v \cr x & \w_{33} \end{pmatrix}:=  \begin{pmatrix} \tiny\begin{pmatrix}\a^2 & 0\cr 0 &\b\end{pmatrix}M\tiny\begin{pmatrix}\a^{-4} & 0\cr 0 &\b^{-2}\end{pmatrix}& v \cr x &  \frac{\w_{33}}{\a} \end{pmatrix} 
 \end{array}.
\]

\item \label{caso} Consider the product group $(\K^\times)^{\esc{2}}\times\Z_2$, which acts on the set $\K^2\setminus\{0\}$ in the following way

\[
 \begin{array}{cc}
      &  [(\K^\times)^{\esc{2}}\times\Z_2]\times(\K^2\setminus\{0\}) \to \K^2\setminus\{0\}\\
      & (\lambda,i)\cdot v :=  \lambda v E_{12}^i .
 \end{array}
\]

\item \label{casa} Let the  group $\K^{\times}$ act on the set $\K^2\setminus\{0\}\times \K^{\times}$, via the action defined by
\[
 \begin{array}{cc}
      &  \K^\times\times[(\K^2\setminus\{0\})
      \times \K^{\times}] \to (\K^2\setminus\{0\})\times \K^{\times}\\
      & \lambda\cdot (x,y,z) :=(\l^2x,\l^2y,\l z  ).
 \end{array}
\]

\end{enumerate}
\end{definition}

\begin{notation}\rm
Let $\mathcal{W}=\left\{ (\w_{ij}) \in \GL_2{(\K)} \,  : \, \w_{12}\w_{21} \neq 0 \right\}$. Observe that,  applying \cite[Corollary 4.6]{CSV1}), the matrices of this set correspond to the structure matrices of the two-dimensional simple evolution algebras.
\end{notation}

\subsubsection{ $\soc(A)$ has the extension property}

We consider first the case in which the socle is generated by two vectors of a natural basis of the algebra.

\begin{theorem}
\label{partwotwo}
Let $A$ be a three-dimensional evolution algebra with zero annihilator.
Assume that $\dim(\soc(A))=2$ and $\soc(A)$ has the extension property.
Then one of the following three excluding possibilities holds:
\begin{enumerate}
\item $\soc(A)$ is a minimal ideal and $\ssi(A)=1$. Equivalently, the structure matrix of $A$ is of the form  $\tiny\begin{pmatrix}M & v \cr 0 & 0\end{pmatrix},$ where $M \in \mathcal{W}$ and $v \neq 0$. Furthermore,
\begin{enumerate}
    \item If $\vert \supp_\B{(e_3^2)}\vert=2$, then the structure matrix of $A$ relative to a suitable basis is \[\tiny{\begin{pmatrix}M & \begin{pmatrix} 1 \cr 1 \end{pmatrix}\cr 0 & 0\end{pmatrix}},\]  where $M\in\mathcal{W}$. This class of algebras are classified by the action of the group $(\K^\times)^{\esc{2}}\times\Z_2$ on $\W$. More precisely, two algebras of this kind, whose structure matrices are \[\tiny{\begin{pmatrix}M & \begin{pmatrix} 1 \cr 1 \end{pmatrix}\cr 0 & 0\end{pmatrix}} \text{ \small and }
    \tiny{\begin{pmatrix}M' & \begin{pmatrix} 1 \cr 1 \end{pmatrix}\cr 0 & 0\end{pmatrix}},\]
    are isomorphic if and only if $M$ and $M'$ are in the same orbit under the action of $(\K^\times)^{\esc{2}}\times\Z_2$
    described in Definition \ref{grfghgth}(\ref{chicas}). A moduli set for this class is 
    $((\K^\times)^{\esc{2}}\times\Z_2 ,\W)$.
    
    \item If $\vert\supp_{\B}{(e_3^2)}\vert=1$, then the structure matrix of $A$ relative to a suitable basis is \[\tiny\begin{pmatrix}M & \begin{pmatrix} 1 \cr 0 \end{pmatrix}\cr 0 & 0\end{pmatrix},\]  where   $M\in \mathcal{W}$. This class of algebras are classified by the group $(\K^\times)^{\esc{2}}\times \K^{\times}$. More precisely,  two algebras of this kind, whose structure matrices are \[\tiny{\begin{pmatrix}M & \begin{pmatrix} 1 \cr 0 \end{pmatrix}\cr 0 & 0\end{pmatrix}} \text{\small and }
    \tiny{\begin{pmatrix}M' & \begin{pmatrix} 1 \cr 0 \end{pmatrix}\cr 0 & 0\end{pmatrix}},\]
    are isomorphic if and only if $M$ and $M'$ are in the same orbit under the action of $(\K^\times)^{\esc{2}}\times\K^{\times}$
    described in Definition \ref{grfghgth}(\ref{chico}). A moduli set for this class is  $((\K^\times)^{\esc{2}}\times\K^\times ,\W)$.
\end{enumerate}

\item $\soc(A)$ is a minimal ideal and $\ssi(A)=2$. Equivalently, $e_3^2 \notin \soc(A)$ and the structure matrix of $A$ is of the form  $\tiny\begin{pmatrix}M & v \cr 0 & \w_{33}\end{pmatrix}$ where $M \in \W$, $\w_{33}\neq 0$ and $v \neq 0$. Furthermore, 
\begin{enumerate}
    \item If $\{1,2\} \subset \supp_\B{(e_3^2)}$, then the structure matrix of $A$ relative to an appropriate basis is
    $\tiny\begin{pmatrix}M & \binom{1}{1}\cr 0 & \w_{33} \end{pmatrix},$
     where $M \in \W$. This class of algebras are classified by the group $\K^{\times} \times \Z_2$. More precisely,  two algebras of this kind, whose structure matrices are \[\tiny{\begin{pmatrix}M & \begin{pmatrix} 1 \cr 1 \end{pmatrix}\cr 0 & \w_{33}\end{pmatrix}} \text{ \small and }
    \tiny{\begin{pmatrix}M' & \begin{pmatrix} 1 \cr 1 \end{pmatrix}\cr 0 & \w_{33}'\end{pmatrix}},\]
    are isomorphic if and only if both matrices are in the same orbit under the action of $\K^\times\times\Z_2$
    described in Definition \ref{grfghgth}(\ref{nino}). A moduli set in this case is 
    $(\K^\times\times\Z_2,\W).$
    
    \item If $1 \notin \supp_\B(e_3^2)$, or $2 \notin \supp_\B(e_3^2)$, then the structure matrix of $A$ relative to a appropriate basis is of the form  $\tiny\begin{pmatrix}M & \binom{1}{0}\cr 0 & \w_{33} \end{pmatrix},$ where $M \in \W$. This class of algebras can be classified by the group $\K^{\times}\times \K^{\times}$. 
    Concretely, two algebras whose structure matrices are \[\tiny{\begin{pmatrix}M & \begin{pmatrix} 1 \cr 0 \end{pmatrix}\cr 0 & \w_{33}\end{pmatrix}} \text{ \small and }
    \tiny{\begin{pmatrix}M' & \begin{pmatrix} 1 \cr 0 \end{pmatrix}\cr 0 & \w_{33}'\end{pmatrix}}\]
    are isomorphic if and only if both matrices are in the same orbit under the action of $\K^\times\times\K^{\times}$
    described in Definition \ref{grfghgth}(\ref{nina}). A moduli set in this case is $(\K^\times\times\K^{\times},\W)$.

\end{enumerate}

\item $\soc(A)$ is the direct sum of two minimal ideals, say $\K u_1$ and $\K u_2$. Then, the structure matrix of $A$ relative to a suitable basis, $B$, is $\tiny\begin{pmatrix} 1 & 0 & \w_{13}\cr 0 & 1 & \w_{23}\cr 0 & 0 &\w_{33}\end{pmatrix}$. Furthermore,

\begin{enumerate}
    \item If $\ssi{(A)}=1$, then the structure matrix relative to $\B$ is $\tiny\begin{pmatrix}\id_{2} & v\cr 0 & 0\end{pmatrix}$, with $v\ne 0$. This class of algebras are classified by the group $(\K^\times)^{\esc{2}}\times\Z_2$. Concretely, two algebras whose structure matrices are \[\tiny\begin{pmatrix}\id_{2} & v\cr 0 & 0\end{pmatrix} \text{\small and } \tiny\begin{pmatrix}\id_{2} & v'\cr 0 & 0\end{pmatrix}\]
    are isomorphic if and only if both matrices are in the same orbit under the action of $(\K^\times)^{\esc{2}}\times\Z_2$ described in Definition \ref{grfghgth}(\ref{caso}).
    The corresponding moduli is $((\K^\times)^{\esc{2}}\times\Z_2,\K^2\setminus\{0\})$.
    
    \item If $\ssi(A)=2$, then the structure matrix of $A$ relative to $\B$ is $\tiny\begin{pmatrix}\id_{2} & v\cr 0 & \w_{33}\end{pmatrix}$, with $v\neq 0$ and $\w_{33} \neq 0$. This class of algebras are classified by the group $\K^{\times}$. More precisely, two algebras of this kind, whose structure matrices are 
    \[\tiny\begin{pmatrix}\id_{2} & v\cr 0 & \w_{33}\end{pmatrix} \text{ \small and } \tiny\begin{pmatrix}\id_{2} & v'\cr 0 & \w_{33}'\end{pmatrix}\]
    are isomorphic if and only if both matrices are in the same orbit under the action $\K^{\times}$ describes in Definition \ref{grfghgth}(\ref{casa}). The moduli for this class is $(\K^\times,\K^2\setminus\{0\}\times\K^\times)$.
    
\end{enumerate}
\end{enumerate}
\end{theorem}
\begin{proof}
We have
$\soc(A)=\sp{(\{e_1,e_2\})}$ with $\{e_i\}_{i=1}^3$ a natural basis.
Now either $\soc(A)$ is a minimal ideal or it is the direct sum of two such ideals.
\begin{enumerate}
    \item Assume that $\soc(A)$ is a minimal ideal of $A$ and that $\ssi(A)=1$. Then it is easy to check that $\soc(A)$ is a simple evolution algebra. Moreover, as $\ssi{(A)}=1$, then $\soc(A)=\soc^2(A)$. This implies that $\soc{\left(A/\soc(A)\right)}=0$. On the other hand, $A/\soc(A)=\K\overline{e_3}$ implies that $e_3^2 \in \soc(A)$. Note that the structure matrix is of the form $\tiny\begin{pmatrix}M & v \cr 0 & 0\end{pmatrix}$, with $M$ the structure matrix of the evolution algebra $\soc(A)$. Applying \cite[Corollary 4.6]{CSV1}) we get that $M \in \W$. 
Conversely, if we have any matrix of the form $\tiny\begin{pmatrix}M & v\cr 0 & 0\end{pmatrix}$, with $v=\tiny\begin{pmatrix}\w_{13} \cr \w_{23} \end{pmatrix}\neq \begin{pmatrix}0 \cr 0 \end{pmatrix}$, $M$ invertible and with nonzero upper right and lower down entries, then $\soc(A)$ is a minimal ideal and $\ssi(A)=1$. Indeed, let $A$ be the evolution algebra and $\B=\{e_1,e_2,e_3\}$ the natural basis such that the structure matrix relative to $\B$ is $\tiny\begin{pmatrix}M & v\cr 0 & 0\end{pmatrix}$. 
We note that $I=\sp{(\{e_1,e_2\})}$ is a minimal ideal of $A$, because the structure matrix of $I$ is $M$ and $M$ is invertible and has nonzero upper right and lower down entries (see \cite[Corollary 4.6]{CSV1}). Therefore $I \subseteq \soc{(A)}$. Now, $\soc(A)\ne A$, since otherwise $A$ would be a direct sum of simple algebras, but as $A$ is not perfect. So, $\soc(A)\ne A$ and $I=\soc(A)$. Furthermore, $\ssi{(A)}=1$ because, as $\soc{(A)}=I$ and
$\soc{\left(A/\soc{(A)}\right)}=\soc(\K \overline{e_3})=0$, we have $\soc^2{(A)}/\soc(A)=0$. 
    \begin{enumerate}
        \item   If $\w_{13}\w_{23}\neq 0$, $v$ can be chosen to be (scaling the basis if necessary)  $\tiny\binom{1}{1}$. Now, we will see that an evolution algebra with structure matrix $\tiny\begin{pmatrix}M & \binom{1}{1}\cr 0 & 0\end{pmatrix}$ is isomorphic to any algebra with structure matrix  $\tiny\begin{pmatrix}M' & \binom{1}{1}\cr 0 & 0\end{pmatrix}$ if and only if $M'=k M$ or $\tiny M'=k \begin{pmatrix} 0 & 1 \cr 1 & 0\end{pmatrix} M  \begin{pmatrix} 0 & 1 \cr 1 & 0\end{pmatrix}$  with $k \in \K^{\times}$ and $k$ is the square of some element. Indeed, let $A$ be an evolution algebra with natural basis $\B=\{e_1,e_2,e_3\}$ and structure matrix relative to $\B$, given by $\tiny\begin{pmatrix}M & \binom{1}{1}\cr 0 & 0\end{pmatrix}$, with $M \in \W$. We assume that there exists another evolution algebra $A'$ with natural basis $\B'=\{e_1',e_2',e_3'\}$ and structure matrix relative to $\B'$ of the form $\tiny\begin{pmatrix}M' & \binom{1}{1}\cr 0 & 0\end{pmatrix}$, with $M\in \W $, such that $A$ and $A'$ are isomorphic. Let $\phi$ be the isomorphism of evolution algebras between $A$ and $A'$. As $\phi(\soc{(A)})=\soc(A')$ then $\phi(e_1)=\a_{1} e_1'+\a_2e_2'$ and $\phi(e_2)=\b_{1} e_1'+\b_2e_2'$. But $\phi(e_1)\phi(e_2)=0$, because $\phi$ is an isomorphism of algebras and $e_1e_2=0$. Therefore $\a_1\b_1(e_1')^2 + \a_2\b_2(e_2')^2=0$. Applying that $M \in \W$ (so in particular $\vert M \vert \neq 0$) we get that $\a_1\b_1=0$ and $\a_2\b_2=0$. This implies that $\a_1=\b_2=0$ or $\a_2=\b_1=0$. Let $\phi(e_3)=\gamma_1e_1'+\gamma_2e_2'+\gamma_3e_3'$. As $\{\phi(e_1), \phi(e_2), \phi(e_3)\}$ is a natural basis of $A'$ then, in any case, $\gamma_1=\gamma_2=0$. Now we deal with the case $\a_1=\b_2=0$. We have that $e_i^2=\w_{1i}e_1 + \w_{2i}e_2$ for $i\in \{1,2\}$ and therefore, applying $\phi$, we get $\phi(e_1^2)=\w_{11}\a_2e_2' + \w_{21}\b_1e_1'$ and $\phi(e_2^2)=\w_{12}\a_2e_2' + \w_{22}\b_1e_1'$. But $\phi(e_1^2)=(\phi(e_1))^2=\a_2^2(e_2')^2$ and $\phi(e_2^2)=(\phi(e_2))^2=\b_1^2(e_1')^2$. Then $(e_1')^2=\frac{\w_{22}}{\b_1}e_1'+\frac{\w_{12}\a_2}{\b_1^2}e_2'$ and $(e_2')^2=\frac{\w_{21}\b_1}{\a_2^2}e_1'+\frac{\w_{11}}{\a_2}e_2'$. Furthermore, as $e_3^2=e_1+e_2$, we have that $\phi(e_3^2)=\a_2e_2'+\b_1e_1'$. But $\phi(e_3^2)=(\phi(e_3))^2=\gamma_3^2(e_3')^2$, and therefore $(e_3')^2=\frac{\b_1}{\gamma_3^2}e_1'+\frac{\a_2}{\gamma_3^2}e_2'$. So, as the structure matrix relative to $\B'$ is of the form $\tiny\begin{pmatrix}M' & \binom{1}{1}\cr 0 & 0\end{pmatrix}$ with $M\in \W $, necessarily $\frac{\b_1}{\gamma_3^2}=1$ and $\frac{\a_2}{\gamma_3^2}=1$. Then,  $\b_1=\a_2=\gamma_3^2$ and $\tiny M'=\frac{1}{\b_1} \begin{pmatrix} 0 & 1 \cr 1 & 0\end{pmatrix} M  \begin{pmatrix} 0 & 1 \cr 1 & 0\end{pmatrix}$. Now, if $\a_2=\b_1=0$ we proceed analogously to the previous case and we obtain that $M'=\frac{1}{a_1}M$ and $\a_1=\gamma_3^2$.
        So, this means that $M$ and $M'$ are in the same orbit under the action of the group $(\K^\times)^{\esc{2}}\times\Z_2$
    described in Definition \ref{grfghgth}(\ref{chicas}). 
    
    \item If $\w_{13}= 0$ or $\w_{23}=0$, $v$ can be chosen to be (scaling the basis if necessary)   $\tiny\begin{pmatrix}1 \cr 0 \end{pmatrix}$. Now, it can be proved that an algebra with structure matrix  $\tiny\begin{pmatrix}M &  \binom{1}{1} \cr 0 & 0\end{pmatrix}$
     is not isomorphic to any algebra with structure matrix
   $\tiny\begin{pmatrix}M & \binom{1}{0}\cr 0 & 0\end{pmatrix}$. Moreover, any algebra with structure matrix $\tiny\begin{pmatrix}M & \binom{1}{0}\cr 0 & 0\end{pmatrix}$ is isomorphic to any algebra with structure matrix  $\tiny\begin{pmatrix}M' & \binom{1}{0}\cr 0 & 0\end{pmatrix}$ if and only if $M'=\tiny \begin{pmatrix} \dfrac{\omega_{11}}{\alpha} & \dfrac{\w_{12}\a}{\beta^2} \cr \dfrac{\w_{21}\b}{\a^2} & \dfrac{\omega_{22}}{\b}\end{pmatrix}$ with $\a, \b \in \K^{\times}$ and $\a$ is the square of some element. In other words, both structure matrices $M$ and $M'$ are in the same orbit under the action of  the group $(\K^\times)^{\esc{2}}\times\K^{\times}$
    described in Definition \ref{grfghgth}(\ref{chico}).
    \end{enumerate} 
     
     \item Suppose that $\soc(A)$ is a minimal ideal and $\ssi(A)=2$. Then  $e_3^2\notin\soc(A)$ ($\w_{33} \neq 0$). Assume that $e_3^2=s+\w_{33}e_3$ with $s\ne 0$ (if $s=0$ then $A= \soc{(A)}$, which is not the case). So $e_3^2=\w_{13}e_1+\w_{23}e_2+\w_{33}e_3$. This means that the structure matrix of $A$ is  $\tiny\begin{pmatrix}M & v \cr 0 & \w_{33}\end{pmatrix}$, with $M \in \W$ and $v=\tiny\begin{pmatrix} \w_{13} \cr \w_{23} \end{pmatrix}\neq \tiny\begin{pmatrix} 0 \cr 0 \end{pmatrix}$. Reciprocally, if we have an evolution algebra with structure matrix $\tiny\begin{pmatrix}M & v \cr 0 & \w_{33}\end{pmatrix}$, where $M\in \mathcal{W}$ and $v \neq 0$, then it is easy to check that $\soc(A)$ is a minimal ideal and $\ssi(A)=2$.   
     \begin{enumerate}
         \item If $\w_{13}\w_{23}\ne 0$ then, scaling the basis if necessary, we get that the structure matrix of $A$ is $\tiny\begin{pmatrix}M & \binom{1}{1}\cr 0 & k\end{pmatrix}$, with $k \in \K^{\times}$ and $M$ is the structure matrix of the simple evolution algebra $\soc(A)$. Reciprocally, if we have an evolution algebra as in the hypothesis of the proposition, with structure matrix $\tiny\begin{pmatrix}M & \binom{1}{1}\cr 0 & \w_{33}\end{pmatrix}$, where $\w_{33} \in \K^{\times}$ and $M \in \W$, then it is easy to check that $\soc(A)$ is a minimal ideal and $\ssi(A)=2$. Furthermore, it can be proved that any evolution algebra with structure matrix $\tiny{\begin{pmatrix}M & \binom{1}{1}\cr 0 & \w_{33}\end{pmatrix}}$ is isomorphic to any algebra with structure matrix $\tiny{\begin{pmatrix}M' & \binom{1}{1}\cr 0 & \w_{33}'\end{pmatrix}}$ if and only if both are in the same orbit under the action of $\K^\times\times\Z_2$
    described in Definition \ref{grfghgth}(\ref{nino}).
         
     \item If $\w_{13}=0$ and $\w_{23}\ne 0$, or $\w_{23}=0$ and $\w_{13}\ne 0$, scaling the basis if necessary, then we get that the structure matrix of $A$ is  $\tiny\begin{pmatrix}M & \binom{1}{0}\cr 0 & \w_{33}\end{pmatrix}$, with $\w_{33} \in \K^{\times}$ and $M$ is the structure matrix of $\soc(A)$ (which is a simple evolution algebra./)
    Conversely, let $A$ be an evolution algebra under the hypothesis of the proposition with structure matrix $\tiny\begin{pmatrix}M & \binom{1}{0}\cr 0 & \w_{33}\end{pmatrix}$, $\w_{33} \in \K^{\times} $ and $M \in \W$. It is straightforward that $\ssi(A)=2$. Now, it is also easy to prove that an evolution algebra with structure matrix $\tiny{\begin{pmatrix}M & \begin{pmatrix} 1 \cr 0 \end{pmatrix}\cr 0 & \w_{33}\end{pmatrix}}$ is isomorphic to an evolution algebra with structure matrix $\tiny{\begin{pmatrix}M' & \begin{pmatrix} 1 \cr 0 \end{pmatrix}\cr 0 & \w_{33}'\end{pmatrix}}$ if and only if both matrices are in the same orbit under the action of $\K^\times\times\K^{\times}$
    described in Definition \ref{grfghgth}(\ref{nina}).
     
     \end{enumerate}
     \item Assume $\soc(A)=\K v_1\oplus\K v_2$. Let $k_1, k_2 \in \K$ be such that $v_1^2=k_1v_1$ and $v_2^2=k_2v_2$. Now, if we consider $u_1=k_1^{-1}v_1$ and $u_2=k_2^{-1}v_2$, then the basis $\{u_1,u_2,e_3\}$ is a natural basis where the $u_i$'s are idempotents. Observe that $u_ie_3=0$, because $u_i\in\soc(A)=\sp(\{e_1,e_2\})$. Therefore $\{u_1,u_2,e_3\}$ is a new natural basis of $A$ and the structure matrix relative to this basis is $\tiny\begin{pmatrix} 1 & 0 & \w_{13}\cr 0 & 1 & \w_{23}\cr 0 & 0 &\w_{33}\end{pmatrix}$. Now we discuss two cases:
    \begin{enumerate}
        \item If $\ssi(A)=1$, or equivalently $\w_{33}=0$, then the structure matrix relative to $\{u_1,u_2,e_3\}$ is of the form $\tiny\begin{pmatrix}\id_{2} & v\cr 0 & 0\end{pmatrix}$ with $v\ne 0$. Two evolution algebras whose structure matrices are 
        $\tiny\begin{pmatrix}\id_{2} & v\cr 0 & 0\end{pmatrix}$ and 
        $\tiny\begin{pmatrix}\id_{2} & v'\cr 0 & 0\end{pmatrix}$ (for $v,v'\ne 0$) are isomorphic if and only if there is a
        $k\in(\K^\times)^{\esc{2}}$ such that $v'=kvE_{12}^i$ (where $i=0,1$). This means that both matrices are in the same orbit under the action of $(\K^\times)^{\esc{2}}\times\Z_2$ described in Definition \ref{grfghgth}(\ref{caso}).
        \item If $\ssi(A)=2$ or, equivalently, $\w_{33}\neq 0$. If $\w_{33}\neq 0$ then $\K e_3$ is not an ideal of $A$, because if $\K e_3\triangleleft A$ then necessarily $\soc(A)=A$, which is not the case. This implies that $\soc(A/\soc(A))\ne 0$ and hence $\soc^{(2)}(A)$ contains strictly $\soc(A)$. So $\ssi(A)=2$. The reciprocal is straightforward. In this case, the structure matrix relative to $\{u_1,u_2,e_3\}$ is $\tiny\begin{pmatrix}\id_{2} & v\cr 0 & \w_{33}\end{pmatrix}$, with $v\ne 0$ and $\w_{33}\neq 0$. Moreover, two evolution algebras whose structure matrices are $\tiny\begin{pmatrix}\id_{2} & v\cr 0 & \w_{33}\end{pmatrix}$ and 
        $\tiny\begin{pmatrix}\id_{2} &  v'\cr 0 & \w_{33}'\end{pmatrix}$ (for $v=\binom{\w_{13}}{\w_{23}} \ne 0,\, v'=\binom {\w_{13}'}{\w_{23}'}\ne 0$ and $\w_{33} \w_{33}' \ne 0$)   are isomorphic if and only if there is an $k\in \K^{\times}$ such that $\w_{13}'=k^2\w_{13}$, $\w_{23}'=k^2\w_{23}$ and $\w_{33}'=k\w_{33}$. In other words, two evolution algebras are isomorphic if and only if both structure matrices are in the same orbit under the action $\K^{\times}$ describes in Definition \ref{grfghgth}(\ref{casa}).
        \end{enumerate}

\end{enumerate}    
\end{proof}


Now we focus in the case in which $\soc(A)=\sp(\{e_1,e_2+e_3\})$, with $\{e_1,e_2,e_3\}$ a natural basis of $A$ and there is no natural basis of $A$ such that $\soc(A)$ is generated by two elements of this natural basis.

\begin{lemma}\label{noab}
Assume that $A$ is a commutative algebra with zero annihilator with a basis $\{u_1,u_2,u_3\}$ whose multiplication table is summarized in
\[u_1^2=u_1u_2=0, \ u_1u_3=\a u_1,\  u_2^2=u_2, \ u_2u_3=\b u_2,\ u_3^2=\sum\g_i u_i,\]
with $\a, \,\b \in \K$.
Then:
\begin{enumerate}
    \item If some of $\b,\g_2$ or $\g_3$ is nonzero then $A$ is not an evolution algebra.
    \item If $\b=\g_2=\g_3=0$ then $A$ is isomorphic to the evolution algebra with structure matrix 
   \begin{equation} \label{mat}
   \tiny \begin{pmatrix}1 & 0 & 0\cr 
    0 & 1 & -1\cr 0 & -1 & 1
    \end{pmatrix}.
    \end{equation}
\end{enumerate}
\end{lemma}
\begin{proof}
Observe that $\a \ne 0$ because $\ann(A)=0$.
Write $u_3^2=\sum_i \g_i u_i$ with $\g_i\in\K$. Consider the inner products $\esc{\cdot , \cdot }_i\colon A\times A\to \K$ ($i=1,2,3$) such that
$xy=\sum_i\esc{x,y}_iu_i$. The matrices of the $\esc{\cdot , \cdot }_i$ are
\[\tiny\begin{pmatrix}0 & 0 & \a \cr 0 & 0 & 0 & \cr \a & 0 & \gamma_1\end{pmatrix},\ \tiny\begin{pmatrix}0 & 0 & 0\cr 0 & 1 & \b \cr 0 & \b & \gamma_2\end{pmatrix}  \text{ \small and }\tiny\begin{pmatrix}0 & 0 &0\cr 0 & 0&0 \cr0&0 &\gamma_3\end{pmatrix}\]
\noindent
respectively.\newline
Then $A$ is an evolution algebra if and only if there is a basis of $A$ orthogonalizing simultaneously the inner products $\esc{\cdot , \cdot}_i$ for $i=1,2,3$.  Suppose that such a basis $\mathcal B$ exists.
Next we prove that, scaling if necessary, we may assume that $u_2$ is in $\mathcal B$. Indeed, if we write ${\mathcal B}=\{f_1,f_2,f_3\}$ then we get that $\rad(\esc{\cdot , \cdot}_1)=\sp(u_2)$, and hence, applying Lemma \ref{grajo} to $u_1$ (taking $B=\sp(\{u_2,u_3\})$), we find that $u_2$ is in $\B$ (up to nonzero scalars). Now:
\begin{enumerate}
    \item If $\b\ne 0$ or $\g_2\ne 0$ then $\rad(\esc{\cdot,\cdot}_2)=\sp(u_1)$, and arguing as before (applying Lemma~\ref{grajo}) we conclude that $u_1\in B$.
    \item If $\b=\g_2=0$ but $\g_3\ne 0$ we have (taking into account again Lemma~\ref{grajo}) that 
    \[\K u_1= \rad(\esc{\cdot,\cdot}_2)\cap \rad(\esc{\cdot,\cdot}_3)=\sp(
    \{f_i\colon \esc{f_i,f_i}_2=0=\esc{f_i,f_i}_3\}).\]
    Hence there is only one $i$ such that $\esc{f_i,f_i}_2=\esc{f_i,f_i}_3=0$ and this $f_i$ is $u_1$ up
    to nonzero scalars. So in this case again $u_1\in \B$.
\end{enumerate}
The conclusion is that in case some of the scalars $\b$, $\g_2$ or $\g_3$ is nonzero we may assume that $u_1,u_2\in\mathcal B$. Now the third element of $\mathcal B$ is a vector $\xi=x u_1+y u_2+z u_3$ and we must have $\esc{\xi,u_i}_j=0$ for $i=1,2$ and $j=1,2,3$. Writing the corresponding equations we get (among other equations) \[\begin{cases}\a z=0\cr y+\b z=0\end{cases}\]
and since $\a\ne 0$ (because the annihilator of $A$ is zero) the only solution is $y=z=0$. But then 
$\xi$ is a multiple of $u_1$ and this is a contradiction. In this case $A$ is not an evolution algebra. But we must analyze the possibility that $\b=\g_2=\g_3=0$. In this case, since $u_2\in\B$, we can find a natural basis
of the form
$\{u_2,x u_1+y u_3, x'u_1+y'u_3\}$ (the corresponding equations are consistent). Then $\sp(\{x u_1+y u_3, x'u_1+y'u_3\})=\sp(\{u_1,u_3\})=:B$ which is an ideal of $A$.
In $B$ (which is an evolution algebra) we have $u_1^2=0$, $u_1u_3=\a u_1\ne 0$ and $u_3^2=\g_1 u_1$. Scaling we get $u_1^2=0$, $u_1u_3=u_1$ and $u_3^2= u_1$. Then, putting $e=u_3$ and $f=u_3-u_1$, we have $ef=u_3(u_3-u_1)=u_1-u_1=0$ and hence the natural basis $\{u_2,e,f\}$ of $A$ has the required structure matrix. 
\end{proof}

\begin{remark} \rm
In the case of an indecomposable commutative algebra (see  \cite[Definition 2.1]{ElduqueLabra}), Lemma~\ref{noab} implies that the algebra is not an evolution algebra. 
\end{remark}

\begin{corollary}\label{pobrereferee} 
Let $A$ be a three-dimensional evolution algebra with zero annihilator. If  $\soc(A)$ is two-dimensional, decomposable, does not have the extension property and $\soc(A)^2\ne 0$, then 
$A$ is decomposable and isomorphic to the evolution algebra whose structure matrix is given in \eqref{mat}.
\end{corollary} 
\begin{proof}
Assume $\soc(A)=\K u_1\oplus \K u_2$. Take into account that if some $u_i^2\ne 0$ we may rescale it to get $u_i^2=u_i$. 
We complete to a basis $\{u_1,u_2,u_3\}$ of $A$. In this basis we get that $u_1u_3=\a u_1$, $u_2u_3=\b u_2$ and $u_3^2=\sum_i \g_i u_i$ for certain $\alpha,   \b,   \gamma_i \in \K$. After rescaling if necessary we have two cases:
\begin{enumerate}
\item $u_i^2=u_i$ ($i=1,2$) and $u_1u_2=0$. 
In this case $\{u_1,u_2,\xi\}$  is a natural basis where $\xi=-\a u_1-\b u_2+u_3$, a contradiction. 
\item $u_1^2=0$, $u_2^2=u_2$. We apply Lemma~\ref{noab}. Since $A$ is an evolution algebra, then it is necessarily isomorphic to the decomposable 
evolution algebra with structure matrix as in \eqref{mat}.
\end{enumerate}
\end{proof}

\subsubsection{Socle of $A$ does not have the extension property}

We proceed with our argumentation. 
If the socle does not have the extension property then, applying Theorem~\ref{cabrao}, there are two possibilities. We study next the case in which 
$\soc(A)=\sp(\{e_1,e_2+e_3\})$, where $\{e_1,e_2,e_3\}$ is a natural basis of $A$, and such that $\soc(A)$ does not have the extension property.

\begin{proposition}\label{allez}
Let $A$ be a tridimensional evolution algebra with $\dim(A^2)=2$. Let  $\{e_i\}_{i=1}^3$ be a  natural basis of $A$. We have 
\begin{enumerate}
    \item If $\soc(A)=\sp(\{e_1,e_2+e_3\})$ then $\soc(A)$ has the extension property if and only if $\{e_2^2,e_3^2\}$ is a linearly dependent set and $e_2^2+e_3^2\ne 0$.
    \item If $\soc(A)=\sp(\{e_1+e_2,e_2+e_3\})$ then $\o_{2i}=\o_{1i}+\o_{3i}$ for each $i$, where $\o_{ij}$ are the structural constants of $A$. If $\ann(A)=0$ and there is no natural basis $\{e_i'\}_{i=1}^3$ of $A$ with 
    $\soc(A)=\sp(\{e_1',e_2'+e_3'\})$, then no vector of $\soc(A)$ is a natural vector (embeddable in a natural basis of $A$).
\end{enumerate}
\end{proposition}

\begin{proof}
We start with the first assertion. If $\soc(A)$ has the extension property then there is a natural basis 
$u_1=xe_1+y(e_2+e_3)$, $u_2=x'e_1+y'(e_2+e_3)$, $u_3=x''e_1+y''e_2+z''e_3$ with
$x,y,x',y',x'',y'',z''\in\K$. Since $u_i$ (for $i=1,2$) is a member of a natural basis, applying \cite[Theorem 3.3.]{BCS} and $\dim(A^2)=2$, we see that the unique possibilities are (up to nonzero scalars and permutations) $u_1=e_1$ and $u_2=e_2+e_3$. In this case, since $u_2u_3=0$ we get $0=(e_2+e_3)(x''e_1+y''e_2+z''e_3)=y''e_2^2+z''e_3^2$ and $y'', z''$ are not both zero. Hence $e_2^2$ and $e_3^2$ are linearly dependent.
Furthermore, we have that $y''\ne z''$, since otherwise $u_3=x''e_1+y''(e_2+e_3)$ would be a linear combination of $u_1$ and $u_2$. Conversely, if there are  scalars $\a,\b\in\K$, $\a\ne \b$, not simultaneously null with $\a e_2^2+\b e_3^2=0$, then $\{e_1,e_2+e_3,\a e_2+\b e_3\}$ is a natural basis.

Let us prove now the second assertion. We have 
$(e_1+e_2)(e_2+e_3)\in\soc(A)$, and hence there are scalars $\a$ and $\b$ such that $\a(e_1+e_2)+\b(e_2+e_3)=(e_1+e_2)(e_2+e_3)=e_2^2$. So 
$\a=\o_{12},\ \b=\o_{32},\ \a+\b=\o_{22}$ and consequently
\[\o_{22}=\o_{12}+\o_{32}.\]
Similarly from the facts that $(e_1+e_2)^2,(e_2+e_3)^2\in\soc(A)$ we conclude that
\[\o_{2i}=\o_{1i}+\o_{3i},\quad (i=1,2,3).\]

Assume now that $z=x(e_1+e_2)+y(e_2+e_3)$  is natural. By Proposition \ref{gutten} we know that $\Supp(z)\ne \{1,2,3\}$. If $\Supp(z)$ has cardinal $1$, for instance $\Supp(z)=\{1\}$, then $x\ne 0$. Also $x+y= 0$, what implies that $y\ne 0$, a contradiction.
Thus $\Supp(z)$ has cardinal $2$. If the natural basis containing $z$ is 
$\{z,z',z''\}$ then, applying Proposition~\ref{gutten}, we get two possibilities (scaling if necessary): 
\begin{enumerate}
    \item $z=e_1+e_2$, $z'=e_1+k e_2$ and $z''=e_3$ (where $k\ne 1$ and $e_1^2+k e_2^2=0$).
    \item \label{plato} $z=e_2+e_3$, $z'=e_2+k e_3$ and $z''=e_1$ (where $k\ne 1$ and $e_2^2+k e_3^2=0$).
    \item $z=e_1+e_3$, $z'=e_1+k e_3$ and $z''=e_2$ (where $k\ne 1$ and $e_1^2+k e_3^2=0$).
\end{enumerate} 

Now, we analyze the first possibility. Since $e_2=\frac{z-z'}{1-k}$, then $e_2+e_3=\frac{1}{1-k}(z-z')+z''$ and this implies that if
$\l=\frac{1}{k-1}$ we have 
\[\soc(A)=\sp(\{e_1+e_2,e_2+e_3\})=\sp(\{z,\l z'+z''\})\]
contradicting the hypothesis that there is no natural basis 
$\{e_i'\}_1^3$ of $A$ for which 
    $\soc(A)=\sp(\{e_1',e_2'+e_3'\})$. 
    
Reasoning in a similar way, we get that  item \eqref{plato} implies a contradiction. For the third possibility, since $z=xe_1+(x+y)e_2+ye_3=e_1+e_3$, then $x=y=1$ and $x+y=0$, so $2=0$. Therefore $\K$ must have characteristic two. In this case, we can see that $\soc(A)=\sp(\{e_1+e_3,e_1+e_2\})=\sp(\{z,\l z'+z''\})$ with $\l=\frac{1}{k-1}$, a contradiction.
\end{proof}

Thus we proceed assuming that $\soc(A)$ is the linear span of 
$\{e_1,e_2+e_3\}$ for a natural basis $\{e_i\}_1^3$ of $A$ and 
$\soc(A)$ does not have the extension property.
A priory, we have two possibilities for $\soc(A)$: either it is a minimal ideal of $A$ or a direct sum of two minimal ideals of $A$. However the next lemma shows that we must focus only in the case in which $\soc(A)$ is minimal.

\begin{proposition}\label{pera}
Let $A$ be a non-degenerate three-dimensional evolution algebra, 
$\soc(A)=\sp(\{e_1,e_2+e_3\})$  and $\soc(A)$ does not have the extension property. 
If $\soc(A)$ is not minimal, then $A$ itself is decomposable and isomorphic to the evolution algebra with structure matrix 
\begin{equation}\label{borracho}
    \tiny\begin{pmatrix}1 & 0 & 0\cr 
    0 & 1 & -1\cr 0 & -1 & 1\end{pmatrix}.
\end{equation}
If $\soc(A)$ is minimal, then $A$ is indecomposable and:
\begin{enumerate}
    \item $\soc(A)=A^2$ and $\ssi(A)=1$.
    \item \label{pera2} There exist $b_0 \in B$ and $\varphi \in \endo{}_{\K}(B)$ such that $A\cong \Au(B,\varphi, b_0,0)$, where $B=\soc(A)$ is a two-dimensional evolution algebra.
\end{enumerate}
\end{proposition}

\begin{proof}
If $\soc(A)$ is not minimal we apply Corollary~\ref{pobrereferee} and we get that the structure matrix relative to a natural basis is as \eqref{borracho}. \\
Let us consider now the case in which $\soc(A)$ is minimal.
Then, since 
$e_1\in\soc(A)$, $e_2^2=e_2(e_2+e_3)$ and $e_3^2=e_3(e_2+e_3)$, we have $e_i^2\in\soc(A)$ for $i=1,2,3$. Hence $0\ne A^2\subset\soc(A)$ and, 
by minimality of the socle, $A^2=\soc(A)$. Then $A/A^2 \cong K $ with zero product, so $\soc(A/\soc(A))=\{0\}$, implying $\ssi(A)=1$. In this case, by Definition~\ref{taza} we have that $A\cong \Au(B,\varphi, b_0,0)$, where $B=\soc(A)$ is a two-dimensional evolution algebra,  $b_0=\w_{13}e_1+\w_{23}(e_2+e_3)$ and the matrix of the linear map $\varphi$ in the basis $\{e_1,e_2+e_3\}$ is $\tiny \begin{pmatrix} 0 & \w_{13} \cr 0 & \w_{23}\end{pmatrix}$. Indeed,   let $\Omega\colon A \to   B\times \K$ be the isomorphism of algebras where $\Omega(e_1)=(e_1,0)$, $\Omega(e_2)=(e_2+e_3,-1)$ and  $\Omega(e_3)=(0,1)$. In this case we have that $A$ is indecomposable. Indeed, if $A=I\oplus J$, then $\soc{(A)}=I$, a contradiction since $\soc(A)$ does not have the extension property.
\end{proof}

\begin{lemma}\label{frambuesa}
Let $A=\Au(B,\varphi,b_0)$ be a  three-dimensional evolution algebra and  $B:=\sp(\{b_1,b_2\})$ with $\{b_1,b_2\}$ a natural basis of $B$.  Then, the following statements are equivalent:
\begin{enumerate}
    \item \label{batata}There exists $b_3\in B $ such that $\varphi=-L_{b_3}$.
    \item \label{patata}$\{(b_1,0),(b_2,0),(b_3,1)\}$ is a natural basis of $A$.
\end{enumerate}
\end{lemma}
\begin{proof} First, we write $u_1=(b_1,0)$ and $u_2=(b_2,0)$. Let us see that
\eqref{batata} implies \eqref{patata}. Indeed, $u_i(b_3,1)=(b_ib_3+\varphi(b_i),0)=(0,0)$ for $i=\{1,2\}$. Therefore $\{u_1,u_2,(b_3,1)\}$ is a natural basis of $A$. Conversely, if $\{u_1,u_2,(b_3,1)\}$ is a natural basis, then $b_ib_3+\varphi(b_i)=0$ for $i\in\{1,2\}$. As $\{b_1,b_2\}$ is a basis of $B$ then $\varphi=-L_{b_3}$.
\end{proof}

Consider a non-degenerate evolution algebra $A$ of dimension $3$, with a socle of dimension $2$, which is an evolution algebra generated by $\{e_1,e_2+e_3\}$, where $\{e_i\}_{i=1}^3$ is a natural basis of $A$.
Further assume that the socle does not have the extension property.
Then, we know by Proposition~\ref{pera} that $\soc(A)=A^2$ and if $A$ is indecomposable, then $\soc(A)$ is not a direct sum of two one-dimensional ideals, hence it is a minimal ideal and $\ssi(A)=1$.

\begin{theorem}\label{parthree}
Let $A$ be a non-degenerate three-dimensional evolution algebra with natural basis $\{e_1,e_2,e_3\}$ and  $B:=\soc(A)=\sp(\{e_1,e_2+e_3\})$. Assume that $\soc(A)$ does not have the extension property. If $A$ is indecomposable, then:
\begin{enumerate}
\item $B$ is an evolution algebra and there is a linear map $\varphi\colon B\to B$ and an element $b_0\in B$ such that $A\cong\Au(B,\varphi,b_0,0)$ (see Proposition \ref{pera} (\ref{pera2})).
\item $B$ is a minimal ideal of $A$, equivalently the unique ideals of $B$ which are $\varphi$-invariant are $0$ and $B$ itself (see Lemma  (\ref{naranja})).
\item  There is no $b\in B$ such that $\varphi=L_b$ (see Lemma \ref{frambuesa}).
\end{enumerate}
The moduli set for this
class of algebra is $(\endo_\K(B)\times B)/\G$ where $\G=\aut(B)\times B\times\K^\times$ (see Proposition \ref{pollo}). Thus
the isomorphy classes of algebras $\Au(B,\varphi,b_0,0)$ (for fixed $B$) are in one-to-one correspondence with the elements of the set $(\endo_\K(B)\times B)/\G$ (see equation \eqref{poy} for the definition of the action).
\end{theorem}

Finally, we must analyze the case in which $\soc(A)=\sp(\{e_1+e_2,e_2+e_3\}$ and there is no natural basis $\{e_1',e_2',e_3'\}$ such that $\soc(A)=\sp(\{e_1',e_2'+e_3'\}$. We know that no vector of the socle is in a natural basis of $A$ by Proposition \ref{allez}.

\begin{theorem}\label{carnosilla}
Let $A$ be a non-degenerate three-dimensional evolution algebra with natural basis $\{e_1,e_2,e_3\}$. Suppose that
$\soc(A)=\sp(\{e_1+e_2,e_2+e_3\})$ and $\soc(A)$ does not have the extension property. Moreover, there is no natural basis $\{e_1',e_2',e_3'\}$ such that $\soc(A)=\sp(\{e_1',e_2'+e_3'\})$. Then
\begin{enumerate}
\item \label{torre} $\soc(A)=A^2$ is a  minimal ideal and  $\ssi(A)=1$.
\item \label{caballo} There exist $b_0 \in B$ and $\varphi \in \endo{}_{\K}(B)$ such that $A\cong \Au(B,\varphi, b_0)$, where $B=\soc(A)$ is a two-dimensional algebra.
\end{enumerate}
The moduli set for this
class of algebra is $(\endo_\K(B)\times B)/\G$ where $\G=\aut(B)\times B\times\K^\times$ (see Proposition \ref{pollo}). Thus
the isomorphism classes of algebras $\Au(B,\varphi,b_0,0)$ (for $B$ fixed) are in one-to-one correspondence with the elements of the set $(\endo_\K(B)\times B)/\G$ (see equation \eqref{poy} for the definition of the action).
\end{theorem}
\begin{proof}
To prove item \eqref{torre}, notice first that it is easy to show that $e_i^2 \in \soc(A)$ for $i=1, 2 , 3$. Then $0 \ne A^2\subset \soc(A)$. This implies $A/\soc(A) \cong \K $ with zero product and therefore $\soc(A/\soc(A))=\{0\}$. Hence, $\ssi(A)=1$. Moreover, if $A^2$ has dimension one, then any nonzero vector is natural (see \cite[Theorem 3.3.]{BCS}), a contradiction. So $\soc(A)=A^2$. If $\soc(A)$ is not minimal, then we apply Corollary~\ref{pobrereferee} and obtain that there exists a natural basis $\{e_1',e_2',e_3'\}$ such that the structure matrix is as \eqref{borracho}, a contradiction because in this case $\soc(A)=\sp(\{e_1',e_2'+e_3'\})$.

For item \eqref{caballo}, by Definition \ref{taza} we have that $A\cong \Au(B,\varphi, b_0)$, where $B=\soc(A)$ is a two-dimensional algebra,  $b_0=\w_{13}(e_1+e_2)+\w_{33}(e_2+e_3)$ and the matrix of the linear map $\varphi$ in the basis $\{e_1+e_2,e_2+e_3\}$ is $\tiny \begin{pmatrix} 0 & \w_{13} \cr 0 & \w_{33}\end{pmatrix}$. Indeed, let $\Omega\colon A \to   B\times \K$ be the isomorphism of algebras where $\Omega(e_1)=(e_1-e_3,1)$, $\Omega(e_2)=(e_2+e_3,-1)$ and  $\Omega(e_3)=(0,1)$. \end{proof}

Let us give an example of an algebra of the kind in Theorem~\ref{carnosilla}. Consider $B=\complex$ as an $\R$-algebra and $A:=\Au(B,\varphi,b_0)$, where $b_0=i$ and $\varphi\colon\complex\to\complex$ is the linear map given by $\varphi(x+iy)=-i(x+y)$ for any $x,y\in\R$. Thus $\varphi(1)=\varphi(i)=-i$. Then $\{(1,1),-(i,1),(0,1)\}$ is a natural basis of $A$
and $B=\complex$ is not an evolution $\R$-algebra since it has no zero divisors. 
Then $B\triangleleft A$ is a minimal ideal since $\complex$ has no proper nonzero ideals (see Lemma \ref{naranja}). Also it is impossible to have $B=\sp(\{e_1',e_2'+e_3'\})$ for some natural basis $\{e_i'\}_{i=1}^3$ of $A$, because $e_1'(e_2'+e_3')=0$ which is not possible in $\complex$. This example proves that the $B$ in Theorem \ref{carnosilla} is not necessarily an evolution algebra.

\subsection{Socle of dimension one.}\label{unidimsoc}

In this final subsection we study the case $\dim(\soc(A))=1$. We distinguish two cases: $\soc(A)^2\neq 0$ and $\soc(A)^2=0$.   For this task we need to introduce two new adjunctions.  The first one will be studied in the following item. 

\subsubsection{Adjunction of type three.}  In order to study the case  $\soc(A)^2\neq0$ we need to define a third type of adjunction of algebras. We construct a new algebra $A$, based on a $\K$ algebra $B$ with a nonzero inner product, such that $\dim(\soc(A))=1+\dim(\soc(B))$. Moreover, we study which conditions are necessary and sufficient for the adjunction to be an evolution algebra.
\begin{definition}
For a given  algebra $B$ with nonzero inner product $\esc{\cdot,\cdot}\colon 
B\times B\to\K$, define $\Av(B,\esc{\cdot , \cdot }):=\K\times B$ with the product
\begin{equation}\label{laqfaltaba}
(\l,b)(\l',b'):=(\l\l'+\esc{b,b'},bb'),\quad \l,\l'\in\K,\ b,b'\in B.
\end{equation}
\end{definition}
Note that the algebra $B$ is not required to be an evolution algebra in this definition.
\begin{remark}\rm 
Let $A$ be a non-degenerate $3$-dimensional evolution algebra of the form 
$A=\K\times B$, where $B$ is a $2$-dimensional algebra and the product of $A$ is given
by $(\l,b)(\l',b')=(\l\l',bb')$. We know that in this case $B$ is a non-degenerate evolution algebra so that 
$\soc(B)\ne 0$ by Proposition \ref{paleta}. Then $\dim(\soc(A))=1+\dim(\soc(B))>1$ and this class of algebras is out of our purpose in this section. This justifies the nonzero inner product hypothesis in  the definition of $\Av(B,\esc{\cdot , \cdot })$.
\end{remark}

\begin{lemma}
If $A=\Av(B,\esc{\cdot , \cdot })$ is a $3$-dimensional algebra then $\soc(A)=\K\times\{0\}$.
\end{lemma}
\begin{proof}
The subspace $\K\times\{0\}$ is an ideal of $A$ and $A(\K\times\{0\})\ne 0$, hence
$\K\times\{0\}\subset\soc(A)$. We claim that $\K\times\{0\}=\soc(A)$: if there is another one-dimensional ideal $\K(\l,b)$, then $(\l,b)(1,0)=(\l,0)\in\K(\l,b)$ implying $b=0$, so $\K(\l,b)=\K(\l,0)$. If $\soc(A)=A$ then there is a minimal ideal $J$ of dimension $2$ in $A$, but if $(\l,b)\in J$ we have  $(1,0)(\l,b)=(\l,0)\in J$. We conclude that either 
$J=\{0\}\times B$ or $\K(1,0)\subset J$ (which is not possible by the minimality of $J$).
Now we check that $\{0\}\times B$ is not an ideal of $A$. Indeed, since $(0,b)(0,b')=(\esc{b,b'},bb')\in\{0\}\times B$ we get that $\esc{\ \cdot, \cdot\ }=0$, a contradiction.
\end{proof}

Next we  prove that any non-degenerate three-dimensional evolution algebra with one-dimensional socle (and such that the square of the socle is non-zero) is isomorphic to an adjunction $\Av(B,\esc{\ \cdot , \cdot\ })$, for suitable $B$ and $\esc{\cdot ,\cdot }$. In fact we prove a more general result, valid for finite dimensional evolution algebras. 
\medskip

\begin{theorem}\label{amsismple} 
Let $A$ be a non-degenerate finite-dimensional evolution algebra such that  $\dim(\soc(A))=1$ and $\soc(A)^2\ne 0$. Then $A\cong \Av(B,\esc{\cdot ,\cdot })$ for a suitable $B$ and $\esc{\cdot ,\cdot }$.
\end{theorem}
\begin{proof}
 Take a nonzero generator $e$ of $\soc(A)$ and
let $\varphi\colon A\to\K$ be the linear map such that $e x=\varphi(x)e$ for any $x\in A$. Observe that 
$\varphi\ne 0$ because $A$ is a non-degenerate evolution algebra. This implies that the kernel $B:=\Ker(\varphi)$ has dimension $\dim(A)-1$. We know that $\soc(A)^2\ne 0$, hence $e^2\ne 0$. Thus $e\notin B$ and, rescaling if necessary, $e$ can be taken to be  an idempotent with $A=\K e\oplus B$, where $eB=0$ and $B$ is a subspace. So, if we multiply $x,y\in B$, we get a part in the socle and a part in $B$. We can formalize this by saying that for any $x,y\in B$ one has 
\[xy=\esc{x,y}e+\beta(x,y),\]
where $\esc{\ \cdot , \cdot\ }\colon B\times B\to\K$ is a bilinear symmetric form and $\beta\colon B\times B\to B$ is a bilinear map. Thus the multiplication of two arbitrary elements $\lambda e+x$ and 
$\lambda' e+x'$ of $A$ would be 
\[(\l e+x)(\l'e+x')=\l\l' e+\esc{x,x'}e+\beta(x,x').\]
Observe that $\beta$ endows $B$ with an algebra structure, so that writing $xy:=\beta(x,y)$ for $x,y\in B$ we get the product defined in \eqref{laqfaltaba}. Summarizing, $A\cong\Av(B,\esc{\ \cdot , \cdot\ })$.  
\remove{As $\soc(A)$ is simple, then
 $e\in B$ and $e^2=0$. So let us take any complementary subspace of $\K e$, that is, a subspace $C$ with $A=\K e\oplus C$ (note that $C$ can be chosen containing an element $f\in B$ such that $ef=f^2=0$ and $\{e,f\}$ are linearly independent). Now we complete to a basis $\{e,f,g\}$ of $A$. We know that $e g\ne 0$ and hence we may take $e g=e$.
On the other hand, there are scalars $\l_i$ and $\mu_i$ such that $fg=\l_1 e+\l_2 f+\l_3 g$ and $g^2=\mu_1 e+\mu_2 f+\mu_3 g$. Consider the vector space decomposition 
$A=\K e\oplus V$, where $V:=\K f\oplus\K g$. Define the inner product $\pi$ as in Lemma \ref{grajo}. The matrix of $\pi$ relative to the basis $\{e,f,g\}$ is 
\[\begin{pmatrix}0 & 0 & 1\cr 0 & 0 & \l_1\cr 1 & \l_1 & \mu_1\end{pmatrix}\] which has rank $2$. Hence any generator of $\rad(\pi)$ is in a natural basis. The radical of $\pi$ is generated by $\l_1 e-f$, hence this element is natural. But $(\l_1 e-f)^2=0$, which contradicts that $A$ is non-degenerate. Hence this possibility can not occur.}
\end{proof}

So far, the algebra $B$ given in Theorem~\ref{amsismple} has not been proved to be an evolution algebra. We address this task in the following lemma. 

\begin{lemma} \label{ruido}
Let $\{(\l_i,b_i)\}_{i=1}^n$ be a basis of $A=\Av(B,\esc{\cdot ,\cdot })$. Then  the following assertions are equivalent:
\begin{enumerate}
\item The set $\{(\l_i,b_i)\}$ is a natural basis of $\Av(B,\esc{\cdot ,\cdot })$.
\item $B$ is an evolution algebra endowed with an inner product $\esc{\cdot ,\cdot }\colon B\times B\to\K$ such that $\esc{b_i,b_j}=-\l_i\l_j$ if $i\ne j$; and the set $\{b_i\}_{i=1}^n$ is pairwise orthogonal.
\end{enumerate}
\end{lemma}

\begin{proof}
Let $\{(\l_i,b_i)\}_{i=1}^n$ be a natural basis of $A$. If $i\ne j$ we have 
\[0=(\l_i,b_i)(\l_j,b_j)=(\l_i\l_j+\esc{b_i,b_j},b_ib_j)\] implying that $\esc{b_i,b_j}=-\l_i\l_j$ and the $b_i$'s are pairwise orthogonal.

Conversely, let $B$ be an evolution algebra
with an inner product $\esc{\cdot ,\cdot }\colon B\times B\to\K$ such that 
$\esc{b_i,b_j}=-\l_i\l_j$ if $i\ne j$.  Then the set $\{(\l_i,b_i)\}_{i=1}^n$ is orthogonal because 
$(\l_i,b_i)(\l_j,b_j)=(\l_i\l_j+\esc{b_i,b_j},b_ib_j)=(0,b_ib_j)=(0,0)$.
\end{proof}

\subsubsection{Isomorphisms between adjunction algebras of type three.}

In this section we deal the study of isomorphism between two algebras $\Av(B,\esc{\cdot ,\cdot })$ and $\Av(B',\esc{\cdot ,\cdot })'$ for suitable $(n-1)$-dimensional evolution algebras $B$ and $B'$.

\begin{theorem}\label{lbnl}
Let $A$ be a non-degenerate evolution algebra with $\dim(A)=n>0$ and with a $1$-dimensional socle whose square is nonzero. Then  $A\cong\Av(B,\esc{\cdot ,\cdot })$ for a suitable $(n-1)$ dimensional evolution algebra $B$. 
Furthermore, $\Av(B,\esc{\cdot ,\cdot })\cong\Av(B',\esc{\cdot ,\cdot }')$ if and only if there is an isometric isomorphism 
$\beta\colon B\to B'$ (isometric in the sense that $\esc{\b(x),\b(y)}'=\esc{x,y}$ for any $x,y\in B$).
\end{theorem}

\begin{proof}
 By Theorem \ref{amsismple} we have $A\cong\Av(B,\esc{\cdot ,\cdot })$ for a suitable $B$ and  $\esc{\cdot ,\cdot }$. Let $\{(\l_i,b_i)\}_{i=1}^n$ be a natural basis of $A$. Applying Lemma \ref{ruido}  we get that the set $\{b_i\}_{i=1}^n$ is pairwise orthogonal. We prove that $\{b_i\}_{i=1}^n$ is a system of generators of $B$. Indeed, given $b\in B$, notice that $(0,b)=\sum_i k_i(\l_i,b_i)$ for some scalars $k_i$ and hence $b=\sum_i k_i b_i$. So $\{b_i\}_{i=1}^n$ is a system of generators of $B$. Therefore, we can select a basis of $B$ by removing one vector of the set $\{b_i\}_{i=1}^n$. After reordering if necessary, we may assume that $\{b_i\}_{i=1}
^{n-1}$ is a natural basis of $B$.

Now we prove the second part of the theorem. Assume that the map $\theta\colon\Av(B,\esc{\cdot , \cdot})\to\Av(B',\esc{\cdot , \cdot}')$ is an isomorphism. Then 
$\theta(1,0)=(k_0,0)$ for a nonzero scalar $k_0$ (because $\theta$ fixes the socle). Also $\theta(0,b)=(\a(b),\b(b))$, where $\a\colon B\to\K$ and $\b\colon B\to B'$ are linear.
Thus $\theta(\l,b)=(\l k_0+\a(b),\b(b))$ for an arbitrary $(\l,b)$. Since $\theta$ is an isomorphim we deduce that  $\b$ is an isomorphism. Furthermore, for any $\l,\l'\in\K$ and any $\b,\b'\in B$ we have:
\[\theta((\l,b)(\l',b'))=\theta(\l\l'+\esc{b,b'},bb')=(k_0(\l\l'+\esc{b,b'})+\a(bb'),\b(bb')),\]
\[\theta(\l,b)\theta(\l',b')=(\l k_0+\a(b),\b(b))(\l' k_0+\a(b'),\b(b'))=\]
\[( (\l k_0+\a(b))(\l' k_0+\a(b'))+\esc{\b(b),\b(b')}',\b(b)\b(b')).\]
We also get
\[k_0(\l\l'+\esc{b,b'})+\a(bb')=(\l k_0+\a(b))(\l' k_0+\a(b'))+\esc{\b(b),\b(b')}'\]
\[\Rightarrow k_0\l\l'+k_0\esc{b,b'}+k_0\a(bb')=\l\l'k_0^2+\l k_0\a(b')+\l'k_0\a(b)+\a(b)\a(b')+\esc{\b(b)\b(b')}'.\]
Then $k_0=1$, and $\esc{b,b'}+\a(bb')=\l \a(b')+\l'\a(b)+\a(b)\a(b')+\esc{\b(b)\b(b')}'$. This forces 
$\l\a(b')+\l'\a(b)=0$ and $\esc{b,b'}+\a(bb')=\a(b)\a(b')+\esc{\b(b),\b(b')}'$.
But the identity $\l\a(b')+\l'\a(b)=0$ implies $\a=0$ so that $\esc{b,b'}=\esc{\b(b),\b(b')}'$ for any $b,b'\in B$.

Conversely, let $\b\colon B\to B'$ be an isometric isomorphism. Define $\theta\colon \Av(B,\esc{\cdot , \cdot})\to \Av(B',\esc{\cdot , \cdot}')$ by $\theta(\l,b):=(\l,\b(b))$. It is easy to check that $\theta$ is an isomorphism of algebras. 
\end{proof}

{
\begin{remark}
By Theorem \ref{lbnl} the isomorphism classes 
of algebras $\Av(B,\esc{\cdot,\cdot})$, where $B$ is fixed, can be described by the moduli set $(\G,\hbox{Sym}^2(B))$, where $\G=\aut(B)$. The action $\G\times \hbox{Sym}^2(B)\to \hbox{Sym}^2(B)$ is $\b\cdot\esc{\quad}=\esc{\quad}'$,
where $\esc{x,y}':=\esc{\b(x),\b(x')}$ for any $x,x'\in B$. Thus the algebras in this class are in one-to-one correspondence with the orbit set $\hbox{Sym}^2(B)/\G$.\medskip
\end{remark}

To end the classification of the three-dimensional evolution algebras, we need to classify the non-degenerate evolution algebras $A$ with one-dimensional socle such that $\soc(A)^2=0$. This will be the goal of our next subsection.

\subsubsection{Adjunction of type four.}

The focus of this subsection are the non-degenerate, three-dimensional evolution algebras $A$ with one-dimensional socle verifying $\soc(A)^2=0$. We present an example of such an algebra below.

\begin{example}\rm
Consider $B$ with natural basis $\{b_1,b_2\}$ such that $b_1^2=0$ and $b_2^2=b_1+b_2$. Let $\varphi\in B^*$ (the usual dual space) be such that $\varphi(b_1)=1$ and $\varphi(b_2)=0$. Consider the inner product $\esc{\cdot , \cdot}\colon B\times B\to\K$ with $\esc{b_1,b_1}=-1=\esc{b_1,b_2}$.
Then, define the algebra $A=\K\times B$ with product 
$$(\l,b)(\l',b')=(\esc{b,b'}+\l\varphi(b')+\l'\varphi(b),bb').$$
The vectors $(1,b_1),(1,b_2),(0,b_1)$ form a natural basis and the structure matrix of $A$ relative to this natural basis is 
$$\begin{pmatrix} 1 & h-1 & -1\cr 0 & 1 & 0\cr -1 & 2-h & 1\end{pmatrix},$$ where $h=\esc{b_2,b_2}$. It can be proved that $\soc(A)=\K(e_1-e_3)$, with $(e_1-e_3)^2=0$.
\end{example}

Notice that if $A$ is a non-degenerate, three-dimensional evolution algebra $A$ with one-dimensional socle such that $\soc(A)^2=0$, then we can write  $\soc(A)=\K e$ and $e^2=0$. So we have a linear map 
$\varphi\colon A\to\K$ such that $ex=\varphi(x)e$ for any $x\in A$. Since $A$ is non-degenerate $\varphi\ne 0$, so $\ker(\varphi)$ has dimension ($\dim(A)-1$). But $\ker(\varphi)=\{x\in A\colon ex=0\}$, the annihilator of $e$ in $A$.  This suggests the following definition of adjunction, which is constructed from an arbitrary $\K$ algebra $B$ with an inner product and a linear form. 
}
\begin{definition}\label{semluz}
Let $B$ be a $\K$-algebra with an inner product $\esc{\cdot , \cdot}\colon B\times B\to\K$ and an element $\varphi\in B^*$, that is, $\varphi\colon B\to\K$ is linear. Then we define in $\K\times B$ the product $$(\l,b)(\l',b')=(\esc{b,b'}+\l\varphi(b')+\l'\varphi(b),bb')$$ for any scalars $\l,\l'\in\K$ and $b,b'\in B$. This algebra will be denoted by $\Aw(B,\esc{\cdot , \cdot},\varphi)$. 
\end{definition}

If $A=\Aw(B,\esc{\cdot , \cdot},\varphi)$ is an evolution algebra, then $B$ is also an evolution algebra. Indeed, if we assume that the
collection $\{(\l_i,b_i)\}_{i=1}^n$ is a natural basis for $A$ then, when $i\ne j$, we have \[0=(\l_i,b_i)(\l_j,b_j)=(\esc{b_i,b_j}+\l_i\varphi(b_j)+\l_j\varphi(b_i),b_ib_j)\] and hence $b_ib_j=0$, which proves that the set
$\{b_i\}_{i=1}^n$ is orthogonal. Next we check that it is a system of generators of the vector space of $B$. Take an arbitrary $b\in B$. Then $(0,b)=\sum_i k_i(\l_i,b_i)$, where $k_i\in\K$. So $b=\sum_i k_i b_i$. This way $B$ is an evolution algebra and a natural basis of $B$ can be obtained by removing some $b_j$ from the collection $\{b_i\}_{i=1}^n$.

\begin{theorem}\label{macabel}
If $A$ is a non-degenerate evolution algebra of dimension $n$, with one dimensional socle and $\soc(A)^2=0$, then there is a $(n-1)$ dimensional evolution algebra $B$ endowed with an inner product $\esc{\cdot , \cdot}\colon B\times B\to\K$ and $\varphi\colon B\to\K$ linear such that $A\cong\Aw(B,\esc{\cdot , \cdot},\varphi)$.
\end{theorem}
\begin{proof}
Let $\soc(A)=\K e$. Then we know that $e^2=0$. The quotient algebra $A/\K e$ is an evolution algebra, hence pick one of its natural basis $\{\bar u_1,\cdots, \bar u_{n-1}\}$. Then $\{e,u_1,\cdots, u_{n-1}\}$ is a basis of $A$ and we have $A=\K e\oplus B$, where 
$B$ is the linear span of $u_1,\ldots,u_{n-1}$. Let $p\colon A\to\K$ be the canonical projection satisfying   $a=p(a)e+b$ for $a \in A$ and $b \in B$. Now, we consider the canonical projection $q\colon A\to B$.
This allows us to define 
$\esc{\cdot , \cdot}\colon B\times B\to\K$ by $\esc{x,y}:=p(xy)$ and 
endow $B$ with a $\K$-algebra structure by defining a product $B\times B\to B$ such that $(x,y)\mapsto x\odot y:=q(xy)$.
We also take into account the linear map $\varphi\colon A\to\K$ such that $xe=\varphi(x)e$ for any $x\in A$ (note that $\varphi=L_e$), and consider the restriction $\varphi\colon B\to\K$. Now, we perform the construction 
$\Aw(B,\esc{\cdot , \cdot},\varphi)$. Finally, we prove that there is an isomorphism $f\colon A\cong\Aw(B,\esc{\cdot , \cdot},\varphi)$. Notice that any $x\in A$ can be writen as $x=\lambda e+b$, with $\lambda\in\K$ and 
$b\in B$. Thus we define $f(\l e+b)=(\l,b)$.
Then, if $y=\l' e+b'$ for $\l'\in\K$ and $b'\in B$, we have $xy=\l eb'+\l' eb+bb'=\l\varphi(b')e+\l'\varphi(b)e+\esc{b,b'}e+b\odot b'$, so that $f(xy)=(\l\varphi(b')+\l'\varphi(b)+\esc{b,b'},b\odot b')$. On the other hand 
$f(x)f(y)=(\l,b)(\l',b')=(\l\varphi(b')+\l'\varphi(b)+\esc{b,b'},b\odot b')$. The isomorphic character of $f$ is trivial. 
\end{proof}

Note that in the conditions of {Theorem \ref{macabel}}, if $A$ happens to have dimension $3$, then there are no minimal ideals of $A$ of dimension $3$ or $2$. Indeed: if $J\triangleleft A$ is minimal and $\dim(J)=2$, then $\K e\cap J=0$ and $A=\K e\oplus J$ (direct sum of ideals), In this case $eJ=0$, whence $eA=0$, which is not possible because $A$ is non-degenerate. If $\dim(J)=3$ then $J=A$ and $A$ is not a minimal ideal since it contains $\K e$. Thus the unique possible minimal ideals of $A$ are the one-dimensional.

Next we want to determine what conditions on $A=\Aw(B,\esc{\cdot,\cdot},\varphi)$ are needed in order to have $\soc(A)=\K(1,0)$.

\begin{lemma}\label{grincho}
If $A=\Aw(B,\esc{\cdot,\cdot},\varphi)$ is $3$-dimensional and non-degenerate,
then  $\soc(A)=\K(1,0)$ if and only if both conditions below are satisfied:
 \begin{enumerate}
     \item\label{gri1} There is no nonzero $b\in B$ such that $b\in B^\bot\cap\ker(\varphi)$ and $Bb\subset\K b$.
     \item\label{gri2} There is no nonzero $b\in \ker(\varphi)$ such that
     $zb=(\varphi(z)+\esc{z,b})b$ for any $z\in B$.
 \end{enumerate}
\end{lemma}
 \begin{proof}
 We prove that if \eqref{gri1} and \eqref{gri2} are satisfied, then $\soc(A)=\K(1,0)$.
  To start with, we check that there is no ideal $\K(0,b)\ne 0$.
 Assume on the contrary that $\K(0,b)\triangleleft A$. Then $0=(1,0)(0,b)=(\varphi(b),0).$
Moreover, for any $z\in B$ we have  $(0,z)(0,b)\in\K(0,b)$, which implies 
 $(\esc{z,b},zb)\in\K(0,b)$. Summarizing:
 \begin{equation}\label{weird}
 \begin{cases} b\in\ker(\varphi)\cr Bb\subset\K b\cr 
 \esc{b,B}=0.\end{cases}
 \end{equation}
  Then the hypothesis imply $b=0$, a contradiction.
 Next we prove that there is no ideal of the form $\K(1,b)$ with $b\ne 0$. Assume on the contrary that such an ideal exists. Then 
 $0=(1,0)(1,b)=(\varphi(b),0)$, whence $b\in\ker(\varphi)$. Also 
 \[(0,z)(1,b)=(\varphi(z)+\esc{z,b},zb)\in\K(1,b),\] and therefore $zb=\a(z)b$ for some linear map 
 $\a\colon B\to\K$ such that 
$\varphi(z)+\esc{z,b}=\a(z)$. Thus 
$\varphi(z)b+\esc{z,b}b= zb$ for any $z\in B$. Consequently, the hypothesis in \eqref{gri2} imply $b=0$, a contradiction. So there are no one dimensional ideals other that $\K(1,0)$. 

Next we prove the converse: if $\soc(A)=\K(1,0)$ then \eqref{gri1} and \eqref{gri2} hold. 
In order to do that, take into account that:
\begin{enumerate}[\rm (i)]
     \item $\K (1,b)$ (with $b\ne 0$) is an ideal of $A$ if and only if $b\in \ker(\varphi)$ and $\esc{b,b'}b + \varphi(b')b=bb'$ for every $b' \in B $,  
    \item $\K(0,b)\triangleleft A $ (for $b\ne 0$) if and only if $b \in \ker(\varphi) \cap B^{\bot}$ and $bB\subset \K b$.
 \end{enumerate}
 So, as $\soc{(A)}=\K(1,0)$,  there is no nonzero $b \in B$ satisfying either of the previous conditions, and hence items \eqref{gri1} and  \eqref{gri2} are verified.
 \end{proof}
 
 \begin{remark}
 Note that both conditions given in the lemma above are satisfied, for instance, if no nonzero $b\in B$ satisfies $Bb\subset\K b$.
 \end{remark}
\subsubsection{Isomorphisms between adjunction algebras of type four.}
So far we know that $3$-dimen\-sional non-degenerate evolution algebras $A$, with one dimensional socle such that $\soc(A)^2=0$, are of the form $\Aw(B,\esc{\cdot , \cdot},\varphi)$ for a $2$-dimensional evolution algebra $B$ satisfying both conditions 
$\eqref{gri1}$ and $\eqref{gri2}$ of Lemma \ref{grincho}. Next we investigate the isomorphism problem for two such algebras $\Aw(B,\esc{\cdot , \cdot},\varphi)$ and $\Aw(B',\esc{\cdot , \cdot}',\varphi')$. If 
 $\theta\colon \Aw(B,\esc{\cdot , \cdot},\varphi)\cong \Aw(B',\esc{\cdot , \cdot}',\varphi')$ is an isomorphism, then $\theta(\K\times \{0\})=\K\times\{0\}$, so 
$\theta(1,0)=(k_0,0)$ for some $k_0\in\K^\times$. Moreover, $\theta(0,z)=(\a(z),\b(z))$ for linear maps 
$\a\colon B\to \K$ and $\b\colon B\to B'$. It can be checked that
\begin{enumerate}
    \item \label{unos}$\varphi'\beta=\varphi$.
    \item \label{doss}$\b$ is an isomorphism from $B$ to $B'$.
    \item  \label{tress} $k_0\esc{z,z'}+\a(zz')=\a(z)\varphi(z')+\a(z')\varphi(z)+\esc{\b(z),\b(z')}'$
for all $z,z'\in B$.
\end{enumerate}
Conversely, it is straightforward to see that if
$\a\colon B\to K$ and $\b\colon B\to B'$ satisfy \eqref{unos}, \eqref{doss} and \eqref{tress}, then $\theta$ is an isomorphism from $\Aw(B,\esc{\cdot , \cdot},\varphi)$ to $\Aw(B',\esc{\cdot , \cdot}',\varphi')$.
Summarizing we claim 
\begin{proposition}
There is an isomorphism $\theta\colon \Aw(B,\esc{\cdot , \cdot},\varphi)\cong \Aw(B',\esc{\cdot , \cdot}',\varphi')$ if and only if there is an isomorphism $\b\colon B\to B'$ and a nonzero $k_0\in\K$ such that
\begin{enumerate}
\item $k_0\esc{z,z'}+\a(zz')=\a(z)\varphi(z')+\a(z')\varphi(z)+\esc{\b(z),\b(z')}'$ for any $z,z'\in B$, and
\item  $\varphi'\b=\varphi$.
\end{enumerate}
\end{proposition}

So fix a $2$-dimensional evolution algebra $B$ and consider the problem of classifying the algebras $\Aw(B,\esc{\cdot , \cdot},\varphi)$ (where the only variables are the inner product and $\varphi$).
Consider the action $\aut(B)\times B^*\to B^*$ such that $\chi\cdot \varphi:=\varphi\chi^{-1}$ for any $\chi\in\aut(B)$ and $\varphi\in B^*$. Now, take $\varphi,\varphi'\in B^*$. If the orbits of $\varphi$ and $\varphi'$ (under the previous action) are different, then there is no isomorphism from $\Aw(B,\esc{\cdot , \cdot},\varphi)$ to $\Aw(B,\esc{\cdot , \cdot}',\varphi')$. If, on the contrary, the orbits of $\varphi$ and $\varphi'$ coincide, then $\varphi'=\varphi\beta$ for some isomorphism $\b\colon B'\to B$.
Thus $\Aw(B,\esc{\cdot , \cdot},\varphi)= \Aw(B,\esc{\cdot , \cdot},\varphi\beta)\cong \Aw(B,\esc{\cdot , \cdot}',\varphi)$.
So we focus on the problem of finding the isomorphism condition for 
\[\Aw(B,\esc{\cdot , \cdot},\varphi)\cong \Aw(B,\esc{\cdot , \cdot}',\varphi),\]
where the unique variable is the inner product in each case. Then the isomorphism exists if and only if there is linear map $\a\colon B\to\K$, and a nonzero $k_0\in\K$, such that
\[k_0\esc{z,z'}+\a(zz')=\a(z)\varphi(z')+\a(z')\varphi(z)+\esc{z,z'}'\] for any $z,z'\in B$. Define the group $\G:=\K\times B^*$ with the product 
$$(k,\a)(k',\a'):=(kk',k\a'+\a),$$
where $k,k'\in\K^\times$, $\a,\a'\in B^*$. The identity element of $\G$ is $(1,0)$. Moreover,
\begin{enumerate}
    \item We identify bilinear forms $B\times B\to\K$ with linear maps $B\otimes B\to\K$. Thus $\esc{\cdot , \cdot}(x\otimes y):=\esc{x,y}$ for $x,y\in B$.
    \item $\mu\colon B\otimes B\to B$ is the product of $B$, that is, $\mu(x\otimes y)=xy$ for $x,y\in B$.
    \item If $\a,\a'\in B^*$ then $\a\otimes\a'$ is the linear map $B\otimes B\to \K$ such that $(\a\otimes\a')(x\otimes y)=\a(x)\a'(y)$, with $x,y\in B$. We can define $\bullet \colon B^* \times B^* \to \K$ by $\a\bullet\a':=\a\otimes\a'+\a'\otimes\a$.
\end{enumerate}
So, there is an action
$$\G\times \hbox{Sym}^2(B)\to \hbox{Sym}^2(B)$$
such that
\begin{equation}\label{rebuz}
(k,\a)\cdot\esc{\cdot , \cdot}:=k\esc{\cdot , \cdot}+\a\mu-\a\bullet\varphi.
\end{equation}

Observe that the linear map in $\eqref{rebuz}$ is an action since $(1,0)\cdot\esc{\cdot , \cdot}=\esc{\cdot , \cdot}$, and
defining $\esc{\cdot , \cdot}':=k\esc{\cdot , \cdot}+\a\mu-\a\bullet\varphi$, we have
$$(k',\a')\cdot [(k,\a)\cdot \esc{\cdot , \cdot}]=(k',\a')\cdot \esc{\cdot , \cdot}'=k'\esc{\cdot , \cdot}'+\a'\mu-\a'\bullet\varphi=$$
$$k'(k\esc{\cdot , \cdot}+\a\mu-\a\bullet\varphi)+\a'\mu-\a'\bullet\varphi=$$
$$kk'\esc{\cdot , \cdot}+(k'\a+\a')\mu-(k'\a+\a')\bullet\varphi=$$
$$(kk',k'\a+\a')\cdot\esc{\cdot , \cdot}=\left((k',\a')(k,\a)\right)\cdot \esc{\cdot , \cdot}.$$
So, we can claim the following.

\begin{proposition}
In the previous setting, we have that the pair $(\G,\hbox{Sym}^2(B))$ is a moduli set for the class of algebras $\Aw(B,\esc{\cdot , \cdot},\varphi)$ with $B$ and $\varphi$ fixed.
\end{proposition}

\section*{Acknowledgments}
The third author was partially supported by Conselho Nacional de Desenvolvimento Cien\-t\'ifico e Tecnol\'ogico (CNPq) grant numbers 304487/2017-1 and 406122/2018-0  and Capes-PrInt grant number 88881.310538/2018-01 - Brazil.
The first and the last two authors are supported by the Junta de Andaluc\'{\i}a  through projects  FQM-336 and UMA18-FEDERJA-119 and  by the Spanish Ministerio de Ciencia e Innovaci\'on   through project  PID2019-104236GB-I00,  all of them with FEDER funds.

\end{document}